\documentclass[12pt,a4paper]{amsart}
\usepackage{amssymb,amscd}
\usepackage{amsfonts}
\usepackage[top=35mm, bottom=35mm, left=30mm, right=30mm]{geometry}
\usepackage[colorlinks=true,citecolor=blue]{hyperref}
\usepackage{mathptmx}
\usepackage{eucal}
\usepackage{graphicx}
\usepackage{mathrsfs}
\usepackage{amssymb}
\usepackage{amsmath}
\usepackage{amsthm}
\usepackage{xcolor}
\usepackage[pagewise]{lineno}\nolinenumbers

\newtheorem{theorem}{Theorem}[section]
\newtheorem{proposition}[theorem]{Proposition}
\newtheorem{lemma}[theorem]{Lemma}
\newtheorem{corollary}[theorem]{Corollary}
\newtheorem*{question1}{Question 1}
\newtheorem*{question2}{Question 2}

\newtheorem*{open question}{Open Question}
\newtheorem*{thma}{Theorem A}
\newtheorem*{thmb}{Theorem B}
\newtheorem*{thmc}{Theorem C}
\newtheorem*{thmd}{Theorem D}

\newtheorem*{conjecture}{Eckmann-Ruelle Conjecture}
\newtheorem{remark}[theorem]{Remark}

\theoremstyle{definition}
\newtheorem{definition}[theorem]{Definition}

\makeatletter

\newcommand{\Rmnum}[1]{\expandafter\@slowromancap\romannumeral #1@}
\makeatother

\begin{document}
\title{Dimensions and entropies for an expansive homeomorphism}
\author{Ercai Chen$^{1}$}
\author{Tassilo K\"{u}pper$^{2}$}
\author{Yunxiang Xie*$^{1}$}
\address{1.School of Mathematical Sciences and Institute of Mathematics, Ministry of Education Key Laboratory of
	NSLSCS, Nanjing Normal University, Nanjing 210023, Jiangsu, P.R.China}
\address{2.Mathematical Institute, University of Cologne,  D-50931 Cologne, Germany}
\email{ecchen@njnu.edu.cn}
\email{kuepper@math.uni-koeln.de}
\email{yxxie20@126.com}
\renewcommand{\thefootnote}{}
\footnotetext{*Corresponding author}
\subjclass[2020]{Primary:  37A15, 37A35, 37C45.}
\keywords{Dimensions; Hyperbolic metrics; Topological entropy;  Brin-Katok entropy; $r$-neutralized entropy; $\alpha$-estimation entropy;  Variational principle;
}
\renewcommand{\thefootnote}{\arabic{footnote}}
\begin{abstract}
For an expansive homeomorphism,  we investigate the relationship among dimension, entropy, and Lyapunov exponents. Motivated by Young's formula for surface diffeomorphisms, which links dimension and measure-theoretic entropy with hyperbolic ergodic measures, we construct the hyperbolic metric with two distinct  Lyapunov exponents $\log b>0>-\log a$. We then examine the relationships between various types of entropies (entropy, $r$-neutralized entropy as well as $\alpha$-estimation entropy) and dimensions. We further prove the Eckmann-Ruelle Conjecture for expansive homeomorphism  with hyperbolic metrics. Additionally, we establish variational principles for these new entropy quantities.
\end{abstract}


\maketitle
\tableofcontents
\section{Introduction}
In classical ergodic theory, entropy and dimension theory are fundamental in unveiling the complexity of dynamical systems. Entropy, a conjugate invariant, quantifies the degree of disorder and information uncertainty within a system. By calculating the entropy of a system, it is possible to quantify the degree of disorder and information uncertainty of the system. Dimension theory describes the geometric structure of the system's phase space, offering insights into its complexity and degrees of freedom.

In the study of smooth dynamical systems, an important object of research is the hyperbolic system, which brings a deep understanding of the relation between entropy and dimensions. Researchers in \cite{BPS99,LY85a,LY85b,Pes77,Pes97,You82} have investigated the intricate relationships among dimension, entropy, and Lyapunov exponents for diffeomorphisms on manifolds, thereby substantially enhancing our understanding of these fundamental concepts and laying the groundwork for further research. 

Let $M$ be a compact smooth Riemannian manifold with the Riemannian metric $d$, $f:M\to M$ be a $C^{1+\alpha}$ diffeomorphism and $\mu$ be a $f$-invariant  ergodic measure. Denote by $d_{\mu}^u(x),d_{\mu}^s(x)$ the  unstable pointwise dimension and stable pointwise dimension of $\mu$, respectively (See   \cite{LY85b} for definitions). Let $\lambda_1>\lambda_2>\cdots>\lambda_{r(x)}$ be distinct Lyapunov exponents of $x$. Recall that an invariant measure $\mu$ is hyperbolic if all Lyapunov exponents are nonzero at $\mu$-almost every point.

A significant contribution to the study of relations among dimension, entropy, and Lyapunov exponents was made by Young \cite{You82} in 1982, who focused on the two-dimensional case.: 
For any hyperbolic ergodic measure $\mu$ with Lyapunov exponents $\lambda_1 > 0 > \lambda_2$,  the relation between  Hausdorff dimension $\dim_H(\mu)$ and the measure-theoretic entropy $h_\mu(f)$ can given by the formula: for $\mu$-a.e. $x\in M$, 
\begin{align}\label{equ 1.1}
	\dim_H(\mu)=\overline{d}_\mu(x)=\underline{d}_\mu(x)=h_\mu(f)\left(\frac{1}{\lambda_1}-\frac{1}{\lambda_2}\right),
\end{align}
where $\overline{d}_\mu(x)$ and $\underline{d}_\mu(x)$ denote the upper and lower pointwise dimension of $\mu$ at $x$. By Ledrappier and  Young's work \cite{LY85b}, we know that $h_\mu(f)=\lambda_1 d_\mu^u(x)$ for $\mu$-a.e. $x\in M$ under the assumption that $\mu$ only admits a positive Lyapunov exponents $\lambda_1.$
In 1985, Eckmann and Ruelle \cite{ER85} discussed the existence of pointwise dimension for hyperbolic invariant measures and  proposed a  famous conjecture:
\begin{conjecture}
	For any hyperbolic measure $\mu$ of $C^{1+\alpha}$ diffeomorphism $f$, the pointwise dimension exists almost everywhere and is constant.
\end{conjecture}
In 1999, for the high-dimensional case, Barreira, Pesin, and Schmeling \cite{BPS99} provided a definitive resolution to the Eckmann-Ruelle Conjecture. They established the precise relationship between the pointwise dimension and both the unstable and stable pointwise dimensions: for $\mu$-a.e. $x\in M$,
\begin{align}\label{equ 1.2}
	\overline{d}_\mu(x)=\underline{d}_\mu(x)=d_{\mu}^u(x)+d_{\mu}^s(x).	
\end{align}

Unlike smooth dynamical systems, topological dynamical systems lack the properties of smoothness. Noticing that the above results hold under the Riemannian metric, we are concerned with how to construct suitable metrics in topological dynamical systems with a similar hyperbolic structure. An interseting question is raised:
\begin{question1}\label{question1}
	Whether the Eckmann-Ruelle Conjecture hold for some metric of topological dynamical systems?
\end{question1}
By a  pair $(\mathcal X,f)$ we will mean  a topological dynamical system (TDS),  where $\mathcal X$ is a compact metrizable  space and $f\colon \mathcal X\rightarrow \mathcal X$ is a homeomorphism. Let $\mathcal D(\mathcal X)$  denote the all metrics compatible with topology on $\mathcal X$. Denote by  $\mathcal{M}(\mathcal X), \mathcal{M}(f), \mathcal{E}(f)$  the sets of Borel probability measures, $f$-invariant probability measures, $f$-ergodic probability measures, respectively.

For continuous maps on metric spaces concerned with the relation between entropy and dimensions, progress has been limited, with only a few specialized results reported \cite{BK83,DZG98,Fat89}. In contrast, expansive dynamical systems have garnered substantial attention, including studies on $\mathbb{Z}$-action \cite{Man79,RR20}, continuous flow \cite{Bow72}, random dynamical systems \cite{CF98,GK20}, symbolic dynamical systems \cite{CX97}, and  $\mathbb{Z}^k$-action \cite{GWZ23,MT19}. A homeomorphism  $f$ of a compact metric space $(\mathcal X,d)$ is said to be expansive if there exists $\epsilon>0$ such that $\sup_{n\in \mathbb{Z}}d(f^nx,f^ny)>  \epsilon$. The property of expansiveness is valuable due to its applications in symbolic dynamics, stability theory, and ergodic theory. Notably, hyperbolic systems provide numerous examples of expansive maps; for instance, the restriction of a diffeomorphism to a hyperbolic set is always expansive \cite{Bow75}. 
In 1987, Fried \cite{Fried 87} utilized Frink's metrization theorem \cite{Fri37} to construct a metric that contracts on stable sets and dilates on unstable sets for expansive maps. Regarding dimension, a significant early result is Ma\~{n}\'{e}'s Theorem \cite{Man79}, which established that a compact metric space admitting an expansive homeomorphism must be finite-dimensional.
\begin{thma}\cite{Man79}
	If $f\colon \mathcal X\to \mathcal X$ is an expansive homeomorphism, then $\mathcal X$ is finite-dimensional.
\end{thma}

Hyperbolicity is an important property in the study of smooth dynamical systems, while in general, topological dynamical systems do not have this property. Since  hyperbolic systems are strongly associated with expansiveness, it seems to be possible to construct hyperbolic metric for an expansive homeomorphism.
In 1989, Fathi \cite{Fat89} showed that  any expansive homeomorphism admits a hyperbolic metric, which combining topology and hyperbolicity from a special perspective and is a partial answer to the Question \ref{question1}.
\begin{thmb}\cite[Theorem 5.20]{Fat89}
	If $f:\mathcal X\to \mathcal X$ is an expansive homeomorphism, then there exists $\alpha\in (1,2)$ and a metric $\tilde{d}:=\tilde{d}_\alpha\in \mathcal D(\mathcal X)$  and numbers $k>1,\epsilon'>0$ such that for any $x,y\in \mathcal X$,
	$$ \max\left\{\tilde{d}(fx,fy),\tilde{d}(f^{-1}x,f^{-1}y)\right\}\geq\min\{k\tilde{d}(x,y),\epsilon'\}.$$	
	Moreover, both $f$ and $f^{-1}$ are Lipschitz for $\tilde{d}$.
\end{thmb}
The metric described in Theorem A is referred to as an \emph{adapted hyperbolic metric} for the expansive homeomorphism, with $\epsilon$ serving as the expansive constant associated with $f$. 
In some sense, $\log \alpha$ can  be regarded as a Lyapunov exponent.
Fathi \cite{Fat89} further proved that any homeomorphism admitting an adapted hyperbolic metric is not only expansive but also characterized by finite topological entropy and Hausdorff dimension, thereby providing a proof of Ma\~{n}\'{e}'s Theorem.
\begin{thmc}\cite[Theorem equ 5.23]{Fat89}
	Let $f:\mathcal X\to \mathcal X$ be an expansive homeomorphism. Suppose that there exists a metric $\tilde{d}\in \mathcal D(\mathcal X)$  and numbers $k>1,\epsilon'>0$ verifying:
	$$\max\{\tilde{d}(fx,fy),\tilde{d}(f^{-1}x,f^{-1}y)\}\geq\min\{k\tilde{d}(x,y),\epsilon'\}$$
	for any $x,y\in \mathcal X$. Then $f$ is expansive and we have: $$\dim_H(\mathcal X,\tilde{d})\leq\overline{{\rm dim}}_B(\mathcal X,\tilde{d})\leq2\frac{h_{top}(f)}{\log k},$$
	where $dim_H(\mathcal X,\tilde{d})$ and $\overline{\dim}_B(\mathcal X,\tilde{d})$ denote the Hausdorff dimension and upper Box dimension of $\mathcal X$.
\end{thmc}
Subsequently, Dai, Zhou, and Geng \cite{DZG98} focused on the dimensions of ergodic measures associated with expansive homeomorphisms and demonstrated the following result:
\begin{thmd}\cite[Theorem 1]{DZG98}
	Let $f:\mathcal X\to \mathcal X$ be a expansive homeomorphism. Suppose that there exists a metric $\tilde{d}\in \mathcal D(\mathcal X)$  and numbers $k>1,\epsilon'>0$ verifying:
	$$\max\{\tilde{d}(fx,fy),\tilde{d}(f^{-1}x,f^{-1}y)\}\geq\min\{k\tilde{d}(x,y),\epsilon'\}$$
	for any $x,y\in \mathcal X$. Then  $$\dim_H(\mu,\tilde{d})\leq\overline{{\rm dim}}_B(\mu,\tilde{d})\leq2\frac{h_\mu(f)}{\log k}$$ holds for any $\mu\in \mathcal{E}(f)$,	where $\dim_H(\mu,\tilde{d})$ and $\overline{{\rm dim}}_B(\mu,\tilde{d})$ denote the Hausdorff dimension and upper Box dimension of $\mu$.
\end{thmd}

In 2019, Meyerovitch and Tsukamoto \cite{MT19} explored the conditions under which an expansive $\mathbb Z^k$-action on a compact metric space implies the existence of finite mean dimension, thereby extending the foundational work of Ma\~{n}\'{e} \cite{Man79}. In parallel, Pacifico and Vieitez \cite{PV20} demonstrated the existence of a metric $d \in \mathcal{D} (\mathcal{X})$ that enables the accurate computation of Lyapunov exponents for expansive homeomorphisms. They showed that these exponents are non-zero for all points in the metric space $(\mathcal X,d)$. However, they also provided a cautionary example, illustrating that this desirable property may not hold when using the original metric. This highlights the complexities and subtleties involved in metric selection. Simultaneously, Roth and Roth \cite{RR20} made significant advancements in developing metrics for linear maps on the torus and non-invertible expansive maps on compact spaces. By constructing these specific metrics, they derived a new lower bound for topological entropy, expressed in terms of the resulting Hausdorff dimensions and Lipschitz constants. This work not only reversed an inequality in \cite{DZG98} but also provided a concise and elegant proof of a well-known theorem on expansive mappings, thereby enriching the theoretical framework in this dynamic research field.

The objective of this paper is to build upon these previous findings and establish a more precise connection. Specifically, we aim to prove additional general results, extending the work in \cite{DZG98,Fat89}. Our first result shows that  any expansive homeomorphism admits a hyperbolic metric with two different  Lyapunov exponents $\log b>0>- \log a$ and  one can precisely determine the relationship between the hyperbolic metric and its image and preimage.
\begin{theorem}\label{thm 1.1}
	If $f:\mathcal X\to \mathcal X$ is an expansive homeomorphism of the compact metric space $\mathcal X$,
	then there is a number $1<\beta\leq 2$ such that for any $a,b\in(1,\beta)$ and for any small enough $\gamma>0$,
	there exists $\tilde{d}\in \mathcal D(\mathcal X)$
	and  $\epsilon'>0$ verifying:
	$$ \max\left\{\frac{\tilde{d}(f^{-1}x,f^{-1}y)}{a-\gamma},\frac{\tilde{d}(fx,fy)}{b-\gamma}\right\}\geq\min\{\tilde{d}(x,y),\epsilon'\}$$
	for any $x,y\in \mathcal X$.
\end{theorem}

Furthermore, the second result shows that the topological entropy and the measure-theoretic entropy of  ergodic measures are  precisely characterized by the Hausdorff dimension and Box dimension under the hyperbolic metric, which is similar to Young's formula (\ref{equ 1.1}) and Barreira, Pesin, and Schmeling's formula (\ref{equ 1.2}). Moreover, it suggests that for any expansive $TDS$, there exists metrics such that  the Eckmann-Ruelle Conjecture also holds.
\begin{theorem}\label{thm 1.2}
	Let $f$ be an  expansive homeomorphism of $ \mathcal X$ and $\tilde{d}$ be the metric  obtained in Theorem \ref{thm 1.1}. Then for any  $\mu\in \mathcal M(f)$, for $\mu$-a.e. $x\in \mathcal X$, we have
	$$d_{\mu}(x,\tilde{d})=h_\mu(f,x)\left(\frac{1}{\log a}+\frac{1}{\log b}\right).$$
	If $\mu\in \mathcal E(f)$, then
	\begin{align*}
		\dim_H(\mu,\tilde{d})=&\dim_B(\mu,\tilde{d})=h_\mu(f)\left(\frac{1}{\log a}+\frac{1}{\log b}\right),\\
		\dim_H(\mathcal X,\tilde{d})=&\dim_B(\mathcal X,\tilde{d})=h_{top}(f)\left(\frac{1}{\log a}+\frac{1}{\log b}\right).
	\end{align*}
\end{theorem}
In addtion, for certain dynamical systems called as  zero-dimensional aperiodic expansive systems, we can remove the condition of $a,b<2$ by using the embeddedness lemma on topological universalities \cite{Kri82} (See Section \ref{Section 6}). 

Recently, the concept of $r$-neutralized entropy has garnered significant attention from scholars, which has an intimate relationship with dimensions. As is well known, the Brin–Katok entropy formula indicates that Bowen balls $\{B_{d}(x,n,r)\}_{n\geq 0}$ can exhibit intricate geometric shapes in the central direction. To facilitate the estimation of the asymptotic measure of sets with diverse geometric configurations, Ovadia and Rodriguez-Hertz \cite{OR24} defined the neutralized Bowen open ball as
$B_{d}(x,n,e^{-nr})=\{y\in \mathcal X: d(f^ix,f^iy)<e^{-nr},0\leq i\leq n\}$
and introduced the neutralized Brin–Katok local entropy. This entropy is calculated by measuring open sets with simpler geometric descriptions. They demonstrated that the neutralized local entropy coincides with the Brin–Katok local entropy almost everywhere. Building on this work, Ovadia \cite{Ova24} further developed the neutralized entropy formula on the unstable manifold. This extension elucidates the relationship between $r$-neutralized unstable entropy, Brin–Katok entropy, and unstable dimensions, highlighting the dependence of $r$. Yang, Chen, and Zhou \cite{YCZ24} successfully established variational principles for the neutralized Bowen topological entropy of compact subsets, closely related to the variational principles for metric mean dimension \cite{LT18}.

To investigate the relationship between $r$-neutralized entropy and dimensions, Dong and Qiao \cite{DQ25} considered a $C^{1+\alpha}$ diffeomorphism $f$ on a closed Riemannian manifold $M$ with the Riemannian metric $d$ and established an equivalence among the $r$-neutralized Brin-Katok entropy, measure-theoretic entropy, and pointwise dimension for any hyperbolic invariant Borel measure $\mu$. They proved that for $\mu$-a.e. $x\in M$,
\begin{align*}
	\overline{h}_{\mu,\tilde{d}}^{BK,r}(f,x)=\underline{h}_{\mu,d}^{BK,r}(f,x)= h_{\mu}(f)+rd_{\mu}(x,d).
\end{align*}
If $\mu$ is hyperbolic ergodic, then for $\mu$-a.e. $x\in M$,
$$\overline{h}_{\mu,d}^{BK,r}(f)=\underline{h}_{\mu,d}^{BK,r}(f)=\overline{h}_{\mu,\tilde{d}}^{BK,r}(f,x)=\underline{h}_{\mu,d}^{BK,r}(f,x)= h_{\mu}(f)+r\dim_{H}(\mu,d).$$
This prompts further inquiry into whether the result remains valid for the hyperbolic metric of expansive homeomorphisms.
The following theorem provides an affirmative answer.
\begin{theorem}\label{thm 1.3}
	Let $f$ be an  expansive homeomorphism of $ \mathcal X$ and $\tilde{d}$ be the metric  obtained in Theorem \ref{thm 1.1}. Given $0<r<\frac{3}{\frac{1}{\log a}+\frac{1}{\log b}}$ ($a,b$ are defined in Section \ref{Section 2}), then for any $\mu \in \mathcal M(f)$, the upper and lower $r$-neutralized (local, Brin-Katok, topology) entropies are equal. If we denote their common value by $h_{\mu,\tilde{d}}^{BK,r}(f,x),h_{\mu,\tilde{d}}^{BK,r}(f), h_{top,\tilde{d}}^{r}(f)$, respectively, then they satisfy the following relations:
	$$d_{\mu}(x,\tilde{d})=h_{\mu,\tilde{d}}^{BK,r}(f,x)\left(\frac{1}{r+\frac{1}{\frac{1}{\log a}+\frac{1}{\log b}}}\right).$$
	Moreover,  if $\mu \in \mathcal E(f)$, then for $\mu$-a.e. $x\in \mathcal X,$ 
	\begin{align*}
		&\dim_{H}(\mu,\tilde{d})=h_{\mu,\tilde{d}}^{BK,r}(f)\left(\frac{1}{r+\frac{1}{\frac{1}{\log a}+\frac{1}{\log b}}}\right),\\
		&\dim_{H}(\mathcal X,\tilde{d})=h_{top,\tilde{d}}^{r}(f)\left(\frac{1}{r+\frac{1}{\frac{1}{\log a}+\frac{1}{\log b}}}\right).
	\end{align*}
\end{theorem}

Additionally, Dong and Qiao \cite[Conjecture 1]{DQ25} proposed a conjecture regarding the existence of analogous variational principles for upper and lower  $r$-neutralized entropy. Specifically, for a  TDS $(\mathcal X,f)$ with the metric $d$ and for any $r>0$, they conjectured that
\begin{align*}
	\overline{h}_{top,d}^{r}(f)=&\sup\{\overline{h}_{\mu,d}^{BK,r}(f):\mu\in \mathcal M(f)\},\\ \underline{h}_{top,d}^{r}(f)=&\sup\{\underline{h}_{\mu,d}^{BK,r}(f):\mu\in \mathcal M(f)\}.
\end{align*}
They verified this conjecture for the topological Markov shift and linear Anosov diffeomorphisms on the torus.  This prompts us to consider whether other systems support this conjecture. Based on current research, we provide a partial answer to this conjecture.
\begin{corollary}\label{cor 1.4}
	Let $f$ be an  expansive homeomorphism of $ \mathcal X$ and $\tilde{d}$ be the metric  obtained in Theorem \ref{thm 1.1}. Then for any $0<r<\frac{3}{\frac{1}{\log a}+\frac{1}{\log b}}$,
	\begin{align*}
		h_{top,\tilde{d}}^{r}(f)=&\left(1+r\left(\frac{1}{\log a}+\frac{1}{\log b}\right)\right)h_{top}(f)\\
		=&\left(1+r\left(\frac{1}{\log a}+\frac{1}{\log b}\right)\right)\sup\{h_{\mu}(f)\colon\mu\in \mathcal M(f)\}\\
		=&\left(1+r\left(\frac{1}{\log a}+\frac{1}{\log b}\right)\right)\sup\{h_{\mu}(f)\colon\mu\in \mathcal E(f)\}\\
		=&\sup\{h_{\mu,\tilde{d}}^{BK,r}(f)\colon\mu\in \mathcal M(f)\}\\
		=&\sup\{h_{\mu,\tilde{d}}^{BK,r}(f)\colon\mu\in \mathcal E(f)\},
	\end{align*}
	where $h_{\mu,\tilde{d}}^{BK,r}(f)$ and $h_{top,\tilde{d}}^{r}(f)$ denote the common value of upper and lower $r$-neutralized local entropy Brin-Katok entropy and $r$-neutralized topological entropy, respectively.
\end{corollary}

Another concept associated with dimension is $\alpha$-entropy, which shares a similar form with $r$-neutralized entropy and was first introduced by Thieullen \cite{Thi91}. In the context of a compact $d$-dimensional Riemannian manifold and a $C^{1+\alpha}$ diffeomorphism $f$, Thieullen developed a novel distance metric known as the $\alpha$-metric, defined as $$d_{n}^{\alpha}(x,y)=\max_{0\leq i\leq n}e^{i\alpha}d(f^ix,f^iy).$$ Thieullen \cite{Thi91} elucidated the relationship between Lyapunov exponents and the local entropies associated with the $\alpha$-metric, thereby extending Pesin's formula \cite{Pes77}.
In 2018, Kawan \cite{Kaw18} proposed a definition of estimation entropy for general systems and provided a lower bound for the estimation entropy of a specific class of systems. Subsequently, Zhong and Chen \cite{ZC21} established Billingsley's theorem and the variational principle for Bowen's estimation entropy, extending the results of \cite{FH12}. Building on these foundational theorems, a natural question arises: 
\begin{question2}\label{ques2}
	For any TDS, are $\alpha$-entropy and $r$-neutralized entropy equal when $\alpha=r$?
\end{question2}
The following Theorem provide a negative answer to this question.
\begin{theorem}\label{thm 1.5}
	Let $f$ be an expansive homeomorphism of $\mathcal X$ and $\tilde{d}$ be the metric  obtained in Theorem \ref{thm 1.1}. Fix
	$0\leq \alpha< \min\{\log a, \log b\}$. Then for any $\mu \in \mathcal M(f)$ and  $\mu$-a.e. $x\in \mathcal X,$
	then for any $\mu \in \mathcal M(f)$, the upper and lower $\alpha$-estimation (local, Brin-Katok, topology) entropies are equal. If we denote their common value by $\widetilde{h_{\mu,\tilde{d}}^{BK,\alpha}}(f,x),\widetilde{h_{\mu,\tilde{d}}^{BK,\alpha}}(f), \widetilde{h_{top,\tilde{d}}^{\alpha}}(f)$, respectively, then they satisfy the following relations:
	\begin{align*}
		d_{\mu}(x,\tilde{d})=\widetilde{h_{\mu,\tilde{d}}^{BK,\alpha}}(f,x)\left(\frac{1}{\log a+\alpha}+\frac{1}{\log b+\alpha}\right).
	\end{align*}
	Moreover, if $\mu \in \mathcal E(f)$, then for $\mu$-a.e. $x\in \mathcal X,$
	\begin{align*}
		\dim_{H}(\mu,\tilde{d})=&\widetilde{h_{\mu,\tilde{d}}^{BK,\alpha}}(f)\left(\frac{1}{\log a+\alpha}+\frac{1}{\log b+\alpha}\right),\\
		\dim_{H}(X,\tilde{d})=&\widetilde{h_{top,\tilde{d}}^{\alpha}}(f)\left(\frac{1}{\log a+\alpha}+\frac{1}{\log b+\alpha}\right).
	\end{align*}
\end{theorem}
However, under hyperbolic metrics, we are able to establish the following variational principle.
\begin{corollary}
	Let $f$ be an expansive homeomorphism of $\mathcal X$ and $\tilde{d}$ be the metric  obtained in Theorem \ref{thm 1.1}.	Fix   $0\leq \alpha< \min\{\log a,\log b\}$. Then we have	
	\begin{align*}
		\widetilde{h_{est,\tilde{d}}^{\alpha}}(f)=\sup\{\widetilde{h_{\mu,\tilde{d}}^{BK,\alpha}}(f)\colon\mu\in \mathcal M(f)\}
		=\sup\{\widetilde{h_{\mu,\tilde{d}}^{BK,\alpha}}(f)\colon\mu\in \mathcal E(f)\}.
	\end{align*}
\end{corollary}

The remainder of this paper is organized  as follows. For an expansive homeomorphism,  we introduce several dimensions  and construct hyperbolic metrics in Section \ref{Section 2}.
Section \ref{Section 3} explores the relationship between classical entropy and dimension. Sections \ref{Section 4} and \ref{Section 5} focus on $r$-neutralized entropies and $\alpha$-estimation entropies, respectively. Finally, in Section \ref{Section 6}, we enhance our results for zero-dimensional expansive aperiodic homeomorphisms.

\section{Preliminary}\label{Section 2}
In this section, we frist present the definitions of various entropies and dimensions, and construct hyperbolic metrics for an expansive homeomorphism. Then, we proceed to the proof of Theorem \ref{thm 1.1}.

\subsection{Entropies and  dimensions}\label{subsection 2.1}
Let $(\mathcal{X},f)$ be a TDS with the metric $d$. Denote the open ball of radius $r$ and centre $x$ with respect to the metric $d$ by
\begin{equation}\label{equ 2.3}
B_{d}(x,r)=\{y\in \mathcal{X}\colon d(x,y)<  r\}.
\end{equation}
For $n\in \mathbb N$ and $x,y\in \mathcal X$, we define the  \emph{n-th Bowen metric} as
\begin{align*}
	d_{n}(x,y)=\max\limits_{0\leq i\leq n}d(f^{i}x,f^{i}y).
\end{align*}
For $m,n\in \mathbb N$,  the \emph{one-sided Bowen ball} and \emph{two-sided Bowen ball}
are respectively denoted by:
\begin{align*}
	&B_{d}(x,n,r)=\{y\in \mathcal X \colon d(f^ix,f^iy)<r,0\leq i\leq n,i\in \mathbb Z\},\\
	&B_{d}(x,-n,m,r)=\{y\in \mathcal X \colon d(f^ix,f^iy)<r,-n\leq i\leq m,i\in \mathbb Z\}.
\end{align*}
\begin{definition}\cite{Wal82}\label{def 2.1}
	Given $m,n\in \mathbb N$ and $r>0,$  a set $E \subseteq \mathcal X$ is said to be an $(n,r)$-spanning set of $\mathcal X$ if
	$\mathcal X=\bigcup_{x\in E}B_{d}(x,n,r)$ and
	a set $F \subseteq \mathcal X$ is said to be a $(-n,m,r)$-spanning set of $\mathcal X$ if
	$\mathcal X=\bigcup_{x\in E}B_{d}(x,-n,m,r)$.
	Denote by $r_{d}(f,n,r)$ the
	smallest cardinality of $(n,r)$-spanning sets of $\mathcal X$ and by $r_{d}(f,-n,m,r)$ the
	smallest cardinality of $(-n,m,r)$-spanning sets of $\mathcal X$.
\end{definition}

\begin{enumerate}
	\item	The \emph{topological entropy} \cite{Wal82} of $f$ is defined as
	\begin{align*}
		h_{top}(f)=\lim_{r\to 0}\limsup_{n \to \infty}\frac{1}{n}\log r_{d}(f,n,r)
		=\lim_{r\to 0}\liminf_{n \to \infty}\frac{1}{n}\log r_{d}(f,n,r).
	\end{align*}
	
	\item  For any $\delta\in (0,1)$,
	the \emph{Katok measure-theoretic entropy} \cite{Kat80} of $f$  with respect to $\mu\in \mathcal E(f)$ is defined as
	\begin{align*}
		h_{\mu}^{K}(f)=\lim_{r\to 0}\limsup_{n \to \infty}\frac{1}{n}\log r_{d}(f,\mu,\delta,n,r)
		=\lim_{r\to 0}\liminf_{n \to \infty}\frac{1}{n}\log r_{d}(f,\mu,\delta,n,r),
	\end{align*}
	where  $r_{d}(f,\mu,\delta,n,r):=\min \{\# K\colon \mu(\bigcup_{x\in K} B_d(x,n,r)  )\geq 1-\delta\}$ and $\delta$ has no effect on the right-hand’s limits.
	
	\item For $\mu$-a.e. $x\in \mathcal X$,	the \emph{Brin-Katok local entropy} \cite{BK83} of $f$ at $x$ with respect to $\mu\in \mathcal M(f)$ is defined as
	\begin{equation}\label{2.4}
		\begin{split}
			h_{\mu}(f,x)=&\lim_{r \to 0} \limsup_{n \to \infty}-\frac{1}{n}\log\mu(B_{d}(x,n,r))\\
			= &\lim_{r \to 0} \liminf_{n \to \infty} -\frac{1}{n}\log\mu(B_{d}(x,n,r)).
		\end{split}
	\end{equation}
\end{enumerate}
By \cite{BK83} we know that for $\mu$-a.e. $x\in \mathcal X$, $h_{\mu}(f,x)$ is $f$-invariant and $\int h_{\mu}(f,x)d\mu(x)=h_{\mu}(f)$. Moreover, if $\mu$ is ergodic, then $h_{\mu}(f)=h_{\mu}(f,x)$ for $\mu$-a.e. $x\in \mathcal X$.
It is worth noting that these entropies are independent of the choice of the compatible metrics on $\mathcal{X}$.

\begin{definition}
	For each $x\in X$, \emph{the upper and the lower pointwise dimension
	of measure} $\mu\in \mathcal M(\mathcal X)$ at $x$ with respect to the metric  $d$ are defined as follows:
	\begin{equation*}
		\begin{split}
			\overline {d}_{\mu}(x,d)&=\limsup\limits_{r \to 0}\dfrac{\log \mu(B_{d}(x, r))}{\log r},\\
			\underline{d}_{\mu}(x,d)&=\liminf\limits_{r \to 0}\dfrac{\log \mu(B_{d}(x,r))}{\log r}.
		\end{split}
	\end{equation*}
	If $\overline {d}_{\mu}(x,d)=\underline {d}_{\mu}(x,d)=d_{\mu}(x,d)$, we call the common value \emph{the pointwise dimension of
	measure} $\mu$ at $x$ with respect to $d$.	
\end{definition}
In the following, we briefly introduce  definitions of  Hausdorff dimension and Box dimension.

\begin{enumerate}
	\item \emph{Hausdorff dimension of $\mathcal X$}$\colon$
	\begin{align*}
		{\rm dim}_H(\mathcal X,d)=\sup\{s\geq0|H^s(\mathcal X,d)=\infty\},
	\end{align*}
	where $H^s(\mathcal X,d)=\lim\limits_{r\to 0}H^s_r(\mathcal X,d)$
	and
	$$H^s_r(\mathcal X,d)=\lim\limits_{r\to 0}\inf\left\lbrace \sum_{n=1}^{\infty}({\rm diam} E_n)^s\Bigg|\begin{array}{l} \mathcal X=\bigcup\limits_{n=1}^{\infty}E_n\;\text{with}\\
		{\rm diam} E_n<r\;\text{for all}\; n\geq1 \end{array}\right\rbrace .$$

	\item \emph{Hausdorff dimension of $\mu\in \mathcal M(\mathcal X)$}$\colon$
	$${\rm dim}_H(\mu,d):=\inf\{{\rm dim}_H(Y,d):\mu(Y)=1,\;Y\subseteq \mathcal X\}.$$
	
	\item \emph{Upper and lower Box dimension of $\mathcal X$}$\colon$
	\begin{align*}
		\overline{{\rm dim}}_B(\mathcal X,d)&=\limsup_{r \to 0}\frac{\log N(\mathcal X,d,r)}{\log \frac{1}{r}},\\
		\underline{{\rm dim}}_B(\mathcal X,d)&=\liminf_{r \to 0}\frac{\log N(\mathcal X,d,r)}{\log \frac{1}{r}},
	\end{align*}
	where $N(\mathcal X,d,r)$ denotes the smallest cardinality of open balls $B_d(x,r)$ needed to cover $\mathcal X$. If $\overline{{\rm dim}}_B(\mathcal X,d)=\underline{{\rm dim}}_B(\mathcal X,d)$, then we call their common value the \emph{Box dimension of  $\mathcal X$} with respect to $d$
	and  denote it by $\dim_B(\mathcal X,d)$.

	\item \emph{Upper and lower Box dimension of $\mu\in \mathcal M(\mathcal X)$}$\colon$
	\begin{align*}
		\overline{{\rm dim}}_B(\mu,d)&=\lim_{\delta \to 0}\inf\{\overline{{\rm dim}}_B(Y,d):Y\subseteq \mathcal X,\;\mu(Y)>1-\delta\},\\
		\underline{{\rm dim}}_B(\mu,d)&=\lim_{\delta \to 0}\inf\{\underline{{\rm dim}}_B(Y,d):Y\subseteq \mathcal X,\;\mu(Y)>1-\delta\}.
	\end{align*}
	If $\overline{{\rm dim}}_B(\mu,d)=\underline{{\rm dim}}_B(\mu,d)$, then we call their common value the \emph{Box dimension of  $\mu$} with respect to $d$
	and  denote it by $\dim_B(\mu,d)$.
\end{enumerate}

The following Lemma clarifies the relations of these dimensions.
\begin{lemma}\cite{DZG98,You82}\label{lem 2.3}
	\rm (1) $\dim_{H}(\mathcal X,d)\leq \underline{\dim}_B(\mathcal{X},d)  \leq \overline{\dim}_B(\mathcal{X},d)$.
	\begin{enumerate}
		\item [(2)]Let $\mu\in \mathcal M(\mathcal X)$ and $\overline{\delta},\underline{\delta}$ be two constants. Then for $\mu$-a.e. $x\in \mathcal X$, if $$ \underline{\delta}\leq\underline{d}_{\mu}(x,d)\leq  \overline{d}_{\mu}(x,d)\leq \overline{\delta},$$  then
		\begin{align*}
		\underline{\delta}\leq 	\dim_{H}(\mu,d)\leq \underline{\dim}_B(\mu,d)\leq \overline{\dim}_B(\mu,d)\leq \overline{\delta}.
		\end{align*}
	\end{enumerate}

\end{lemma}

\subsection{The construction of hyperbolic metrics}
A homeomorphism $f$ of the compact metric space $(\mathcal X,d)$ is said to be \emph{expansive} if there exists $\epsilon>0$ satisfying
$\sup_{n\in \mathbb{Z}}d(f^nx,f^ny)>  \epsilon$.
Throughout the paper, let $f\colon \mathcal X \rightarrow \mathcal X$ be an expansive homeomorphism.

\begin{definition}
	Let $\epsilon$ be an expansive constant for $f$.	For any $x,y\in \mathcal X$, we define $n^+(x,y)$ and $n^-(x,y)$ by$\colon$
	\begin{align*}
		n^{+}(x,y)=\left\{\begin{array}{cc}
			\infty, & {\rm if~}\sup\limits_{n\geq0}d(f^n x,f^ny)\leq  \epsilon \\
			\min\left\{n\geq 0:d(f^n x,f^ny)> \epsilon\right\}, & {\rm if~}\sup\limits_{n\geq0}d(f^n x,f^ny)>  \epsilon
		\end{array}
		\right.
	\end{align*}
	and
	\begin{align*}
		n^-(x,y)=\left\{\begin{array}{cc}
			\infty, & {\rm if~}\sup\limits_{n\leq0}d(f^{n} x,f^{n}y)\leq  \epsilon \\
			\min\left\{n\geq 0:d(f^{-n} x,f^{-n}y)> \epsilon\right\}, & {\rm if~}\sup\limits_{n\leq0}d(f^{n} x,f^{n}y)> \epsilon.
		\end{array}
		\right.
	\end{align*}
\end{definition}

Inspired by the idea of  \cite{Fat89},  $n^-(x,y),n^+(x,y)$ can be estimated by some constant under some certain condition.
\begin{lemma}\label{lem 2.5}
	There exists $m\in \mathbb N$ such that for any $x,y\in \mathcal X$, if $d(x,y)>\frac{\epsilon}{2}$, then one has $\min\{n^-(x,y),n^+(x,y)\}\leq m$.
\end{lemma}

\begin{proof}
	Define $n(x,y)$ by
	\begin{align*}
		n(x,y)=\left\{\begin{array}{cc}
			\infty, & {\rm if~}x=y, \\
			\min\{n_0\in \mathbb N:\max\limits_{|n|\leq n_0}d(f^nx,f^ny>\epsilon\}, & {\rm if~}x\neq y.
		\end{array}
		\right.
	\end{align*}
	Following the Fathi's argument in \cite{Fat89}, there exists an integer $m$ such that if $d(x,y) > \frac{\epsilon}{2}$, then $n(x,y) \leq m$. Since $\min\{n^-(x,y),n^+(x,y)\} = n(x,y)$, we finish the proof.
\end{proof}

Actually, such $m$ is a \emph{uniformly expansive constant} \cite{URM22} that relates expansive and uniformly expansive.
If fixing such $m$ in Lemma \ref{lem 2.5}, then
there exists $\beta:=\beta(m)\in (1,2]$ such that $\beta^m\leq 2$.
For any $a, b\in (1,\beta]$, we define a function $\rho\colon \mathcal X\times \mathcal X\to [0,\infty)$ by
\begin{align}\label{equ 2.5}
	\rho(x,y)=\max\{a^{-n^-(x,y)},b^{-n^+(x,y)}\}.
\end{align}
Clearly, we have
\begin{enumerate}
	\item $\rho(x,y)=\rho(y,x)$.
	
	\item $\rho(x,y)=0$ if and only if $x=y$.
\end{enumerate}
However, $\rho$ may be not a metric on $\mathcal X$,
but  the following Lemma shows that $\rho$ is compatible with  the topology of $\mathcal X$.
\begin{lemma}
	The function $\rho$ is compatible with  the topology of $\mathcal X$.
\end{lemma}

\begin{proof}
	Notice that for any given $r\in (0,1)$, there exists $m(r)\in \mathbb N$ such that for any $x,y\in \mathcal X$, if $d(x,y)>r, $ then
	$\min\{n^-(x,y),n^+(x,y)\} \leq m(r)$ \cite[Proposition equ 5.23.2]{URM22}, which implies that
	$$\rho(x,y)\geq \min \{a^{-m(r)},b^{-m(r)}\}.$$ Take $0<r'<\min \{a^{-m(r)},b^{-m(r)}\}$. Then
	$B_{\rho}(x,r')\subseteq B_{d}(x,r)$.
	
	We remain to show that  for any $x\in \mathcal  X$ and $\eta\in (0,1)$,  there exists $\eta'>0$ such that $B_{d}(x,\eta')\subseteq B_{\rho}(x,\eta)$.
	Take $n=\max\left\{[\frac{-\log \eta}{\log a}],[\frac{-\log \eta}{\log b}]\right\}+1$ such that $\max\{a^{-n},b^{-n}\}<\eta$.
	Define the set
	$$V=\bigcap_{|k|\leq n}f^{-k}B_{d}(f^kx,\epsilon),$$
	where $\epsilon$ is the expansive constant of $f$. Clearly, $V$ is open and thus there exists $\eta'>0$ such that $B_{d}(x,\eta')\subseteq V$ since $x\in V$.
	Pick $y\in B_{d}(x,\eta')$. Then we have
	$d(f^kx,f^ky)<\epsilon$ for any $-n\leq k\leq n.$
	It implies that $\min\{n^-(x,y),n^+(x,y)\}>n$
	and $\rho(x,y)\leq \max \{a^{-n},b^{-n}\}<\eta$,
	which proves that
	$B_{d}(x,\eta')\subseteq B_{\rho}(x,\eta)$.
\end{proof}

We also need the following result for the construction of  hyperbolic metrics.
\begin{lemma}\label{lem 2.7}
	For any $x,y,z\in \mathcal X$,
	$$\rho(x,y)\leq  \max\{a^m,b^m\}\cdot \max\{\rho(y,z),\rho(x,z)\}\leq 2\max\{\rho(y,z),\rho(x,z)\}.$$
\end{lemma}

\begin{proof}
	Without loss of generality, we assume that  $x\neq y$. We divide the proof into three cases.\\
	{\bf Case 1.} If $n^+(x,y)=\infty$ and $n^-(x,y)<\infty$, then $\rho(x,y)=a^{-n^-(x,y)}$.
	
	Since
	$d(f^{-n^-(x,y)}(x),f^{-n^-(x,y)}(y))>\epsilon,$ by the triangular inequality we must have
	$$d(f^{-n^-(x,y)}(z),f^{-n^-(x,y)}(y))>\frac{\epsilon}{2} \quad  \text{or} \quad d(f^{-n^-(x,y)}(z),f^{-n^-(x,y)}(x))>\frac{\epsilon}{2}.$$
	Now we suppose that $d(f^{-n^-(x,y)}(z),f^{-n^-(x,y)}(y))>\frac{\epsilon}{2}$. By Lemma \ref{lem 2.5} we have three subcases to discuss.\\
	(A) If  $n^-(f^{-n^-(x,y)}(z),f^{-n^-(x,y)}(y))=\infty$, then $n^+(f^{-n^-(x,y)}(z),f^{-n^-(x,y)}(y))\leq m.$
	\begin{itemize}
		\item [(A1)] If $n^-(z,y)=\infty$,
		then $n^+(z,y)<\infty$. We get that
		
		$$d_{m}(f^{-n^-(x,y)}(z),f^{-n^-(x,y)}(y))>\epsilon.$$
		It follows that $-n^-(x,y)+m\geq0$ and $0<n^+(z,y)\leq-n^-(x,y)+m$. Thus
		$$\rho(z,y)=b^{-n^+(z,y)}\geq b^{n^-(x,y)-m}\geq b^{-m}.$$
		So that
		$$\rho(x,y)=a^{-n^-(x,y)}\leq1\leq b^m\rho(z,y).$$
		
		\item [(A2)]	If $n^-(z,y)<\infty$, then we have
		$n^-(z,y)\leq n^-(x,y)$ since $$n^-(f^{-n^-(x,y)}(z),f^{-n^-(x,y)}(y))=\infty.$$
		So
		$$\rho(x,y)\leq a^{-n^-(z,y)}\leq
		\rho(y,z)\leq\max\{a^m,b^m\}\rho(y,z).$$
	\end{itemize}
	(B) If
	$n^+(f^{-n^-(x,y)}(z),f^{-n^-(x,y)}(y))=\infty,$ then
	$n^-(f^{-n^-(x,y)}(z),f^{-n^-(x,y)}(y))\leq m.$ It follows that
	$n^-(z,y)\leq n^-(x,y)+m$. Thus $$\rho(z,y)\geq a^{-n^-(z,y)}\geq
	a^{-n^-(x,y)-m}=a^{-m}\rho(x,y).$$\\
	(C) If
	$\max\{n^-(f^{-n^-(x,y)}(z),f^{-n^-(x,y)}(y) ,n^+(f^{-n^-(x,y)}(z),f^{-n^-(x,y)}(y))\}<\infty,$
	then
	$$\min\{n^-(f^{-n^-(x,y)}(z),
	f^{-n^-(x,y)}(y)),n^+(f^{-n^-(x,y)}(z),f^{-n^-(x,y)}(y))\}\leq m.$$
	The proof of the case for
	$$n^-(f^{-n^-(x,y)}(z),f^{-n^-(x,y)}(y))\leq m$$  is
	similar to {\bf Case 1}(B), then we omit the proof. Consider the case that
	$$n^-(f^{-n^-(x,y)}(z),f^{-n^-(x,y)}(y))>m.$$ Then
	$$n^+(f^{-n^-(x,y)}(z),f^{-n^-(x,y)}(y))\leq m.$$ Notice that there exists $0\leq i\leq m$ such that
	$$d(f^{i-n^-(x,y)}(z),f^{i-n^-(x,y)}(y))> \epsilon.$$ If $i-n^-(x,y)\leq0$, then
	$$n^-(z,y)\leq n^-(x,y)-i\leq n^-(x,y).$$ Thus $$ \rho(z,y)\geq
	a^{-n^-(z,y)}\geq a^{-n^-(x,y)}=\rho(x,y).$$ If $i-n^-(x,y)>0$, then
	$$n^+(z,y)\leq i-n^-(x,y)\leq m-n^-(x,y).$$ Hence,
	\begin{align*}
		b^m\rho(z,y)\geq b^{m-n^+(z,y)}\geq b^{n^-(x,y)}\geq 1\geq \rho(x,y).
	\end{align*}\\
	{\bf Case 2.} If $n^+(x,y)<\infty$ and $n^-(x,y)=\infty$, then $\rho(x,y)=b^{-n^+(x,y)}$.
	
	The proof is similar to {\bf Case 1} by replacing $f$ and $b$ with $f^{-1}$ and $a$, respectively.\\
	{\bf Case 3.} If $\max\{n^+(x,y),n^-(x,y)\}<\infty$, we assume that $\rho(x,y)=b^{-n^+(x,y)}$
	for simplicity.
	
	By definition,
	$$d(f^{-n^-(x,y)}(x),f^{-n^-(x,y)}(y))> \epsilon \quad \text{and}\quad d(f^{n^+(x,y)}(x),f^{n^+(x,y)}(y))>\epsilon.$$
	Once more, we may invoke the triangle inequality:
	$$d(f^{n^+(x,y)}(z),f^{n^+(x,y)}(y))> \frac{\epsilon}{2}\quad\hbox{or}\quad d(f^{n^+(x,y)}(z),f^{n^+(x,y)}(x))>\frac{\epsilon}{2},$$
	and
	$$d(f^{-n^-(x,y)}(z),f^{-n^-(x,y)}(y))> \frac{\epsilon}{2}\quad\hbox{or}\quad d(f^{-n^-(x,y)}(z),f^{-n^-(x,y)}(x))> \frac{\epsilon}{2}.$$
	There are 4 possibilities as follows:
	\begin{itemize}
		\item [(a)] $d(f^{n^+(x,y)}(z),f^{n^+(x,y)}(y))>\frac{\epsilon}{2}$ and
		$d(f^{-n^-(x,y)}(z),f^{-n^-(x,y)}(y))>\frac{\epsilon}{2}.$
		
		\item [(b)] $d(f^{n^+(x,y)}(z),f^{n^+(x,y)}(y))>\frac{\epsilon}{2}$ and
		$d(f^{-n^-(x,y)}(z),f^{-n^-(x,y)}(x))>\frac{\epsilon}{2}$.
		
		\item [(c)] $d(f^{n^+(x,y)}(z),f^{n^+(x,y)}(x))>\frac{\epsilon}{2}$ and $d(f^{-n^-(x,y)}(z),f^{-n^-(x,y)}(x))>\frac{\epsilon}{2}.$
		
		\item [(d)] $d(f^{n^+(x,y)}(z),f^{n^+(x,y)}(x))>\frac{\epsilon}{2}$ and $d(f^{-n^-(x,y)}(z),f^{-n^-(x,y)}(y))>\frac{\epsilon}{2}$.
	\end{itemize}
	We only consider the first case.
	
	(a1) If $n^+(z,y)=\infty$, then by Lemma \ref{lem 2.5} again, we have
	$$n^-(f^{n^+(x,y)}(z),f^{n^+(x,y)}(y))\leq m.$$
	Similar to the {\bf Case 1}(A), we get that $$a^m\rho(z,y)\geq
	a^{n^+(x,y)}\geq1\geq\rho(x,y).$$
	
	(a2) If $n^-(z,y)=\infty$, then $n^+(z,y)<\infty$.
	Similar with {\bf Case 1}(B), we get that
	$$b^m\rho(z,y)\geq b^{-n^+(x,y)}=\rho(x,y).$$
	
	(a3)
	If $\max\{n^-(z,y),n^+(z,y)\}<\infty$, then
	$$n^-(f^{n^+(x,y)}(z),f^{n^+(x,y)}(y))\leq n^+(x,y)+n^-(z,y)<\infty.$$
	Similar to the above cases, we consider the case that
	$$n^+(f^{n^+(x,y)}(z),f^{n^+(x,y)}(y))<\infty.$$  Then
	$$\min\{n^+(f^{n^+(x,y)}(z),f^{n^+(x,y)}(y)),n^-(f^{n^+(x,y)}(z),f^{n^+(x,y)}(y))\}\leq
	m.$$
	\begin{itemize}
		\item [(i)] If $n^-(f^{n^+(x,y)}(z),f^{n^+(x,y)}(y))\leq m$, by considering two subcases, we have
		$\rho(x,y)\leq b^m\rho(z,y).$
		
		\item [(ii)] If $n^+(f^{n^+(x,y)}(z),f^{n^+(x,y)}(y))\leq m$,  then a direct observation is that  $n^+(z,y)\leq n^+(x,y)+m.$
		Thus
		$b^m\rho(z,y)\geq \rho(x,y).$
	\end{itemize}
	
	In light of the aforementioned cases, we always have
	$$ \rho(x,y)\leq \max\{a^m,b^m\}\cdot\max\{\rho(z,y),\rho(z,x)\}\leq 2\max\{\rho(z,y),\rho(z,x)\}.$$
\end{proof}

For some certain non-negative real-valued function on $\mathcal X\times \mathcal X$,
Frink's metrization Theorem \cite{Fri37}
guarantees the existence of
the metric compatible with the topology of $\mathcal X$.
\begin{lemma}\label{lem 2.8}
	Let $\mathcal X$ be a nonempty set. If the function $\rho:\mathcal X\times \mathcal X\to [0,\infty)$
	satisfies the following three conditions$\colon$for any $x,y,z\in \mathcal X$,
	\begin{enumerate}
		\item  $\rho(x,y)=0$ if and only if $x=y;$
		
		\item $\rho(x,y)=\rho(y,x);$
		
		\item $\rho(x,y)\leq 2\max\{\rho(z,y),\rho(z,x)\};$
	\end{enumerate}
	then there exists a metric  $D:\mathcal X\times \mathcal X\to [0,\infty)$ such that for any $x,y\in \mathcal X,$
	$$D(x,y)\leq\rho(x,y)\leq 4D(x,y).$$
\end{lemma}

\begin{remark}\label{rem 2.9}
	If we modify the third condition of Lemma \ref{lem 2.8} to $\rho(x,y)\leq K\max\{\rho(z,y),\rho(z,x)\}$, where $K>2$, then we are failed to construct the metric following Frink's approach. Schroeder  has presented a counterexample in \cite{Sch06}.	This is why we
	impose the condition  $\max\{a^m,b^m\}\leq 2$.
\end{remark}

Based on this result, we can  construct a metric $D\in \mathcal D(\mathcal X)$  satisfying
\begin{align}\label{equ 2.6}
	D(x,y)\leq\rho(x,y)\leq 4D(x,y),\quad x,y\in \mathcal X.
\end{align}
We present a basic property for the metric $D$.

\begin{lemma}\label{lem 2.10}
	For any $ x,y\in \mathcal X$ and $n\in \mathbb N$, if
	$$\max\limits_{|i|\leq
		n-1}\{D(f^{i}x,f^{i}y)\}\leq \frac{1}{4}\min\{\frac{1}{a},\frac{1}{b}\},$$ then
	$$ \max\left\{\frac{D(f^{n}x,f^{n}y)}{b^{n}},\frac{D(f^{-n}x,f^{-n}y)}{a^{n}}\right\}\geq \frac{1}{4}D(x,y).$$
\end{lemma}

\begin{proof}
	Since $\rho(x,y)\leq 4D(x,y)$, we have
	$$\max\limits_{|i|\leq
		n-1}\{\rho(f^{i}x,f^{i}y)\}\leq \min\{\frac{1}{a},\frac{1}{b}\}.$$
	It is easy to see that for any $|i|\leq
	n-1$,
	$ \rho(f^{i}x,f^{i}y)\leq \min\{a^{-1},b^{-1}\}$.
	Then we have
	\begin{align}\label{equ 2.7}
		\min\{n^+(f^{i}x,f^{i}y),n^-(f^{i}x,f^{i}y)\}\geq 1.
	\end{align}
	Noting that under such condition,
	$$n^-(f^{-1}x,f^{-1}y)= n^-(x,y)-1, n^+(fx,fy)= n^+(x,y)-1,$$
	then the inequality (\ref{equ 2.7}) implies that $\min\{n^+(x,y),n^-(x,y)\}\geq n$.
	If  $\rho(x,y)=a^{-n^-(x,y)}$, then
	$$ \rho(f^{-n}x,f^{-n}y)\geq a^{- n^-(f^{-n}x,f^{-n}y)}= a^{
		-n^-(x,y)+n}=a^n\rho(x,y).$$
	If  $\rho(x,y)=b^{- n^+(x,y)}$, then
	$$ \rho(f^nx,f^ny)\geq b^{- n^+(f^nx,f^ny)}= b^{
		-n^+(x,y)+n}=b^n\rho(x,y).$$
	Thus
	\begin{align*}
		\max\left\{\frac{\rho(f^nx,f^ny)}{b^n},\frac{\rho(f^{-n}x,f^{-n}y)}{a^n}\right\}\geq
		\rho(x,y).
	\end{align*}
	Since $D(x,y)\leq\rho(x,y)\leq 4D(x,y)$, we have
	\begin{align*} \max\left\{\frac{D(f^{n}x,f^{n}y)}{b^{n}},\frac{D(f^{-n}x,f^{-n}y)}{a^{n}}\right\}\geq \frac{1}{4}D(x,y).\end{align*}
\end{proof}

We are ready to present the proof of Theorem \ref{thm 1.1}.
\begin{proof}[Proof of Theorem \ref{thm 1.1}] Take the metric $D$  defined in the above. Then
	for any $a,b\in(1,\beta]$ and for any small enough $\gamma>0$,
	there exists  $n_0\in \mathbb N$ satisfying
	\begin{align}\label{equ 2.8}
		k_1=4^{-\frac{1}{n_0}}a> a-\gamma, k_2=4^{-\frac{1}{n_0}}b> b-\gamma,
	\end{align}
	where	$n_0$ only depends on the choice of $a,b,\gamma.$
	Using the Mather's trick \cite{Mat82}, we define a metric $\tilde{d}$ by
	\begin{align}\label{equ 2.9}
		\tilde{d}(x,y)=\max_{0\leq i\leq
			n_0-1}\left\{\max\{\frac{D(f^{-i}x,f^{-i}y)}{k_1^{i}},\frac{D(f^{i}x,f^{i}y)}{k_2^{i}}\}\right\}.
	\end{align}
	It is easy to check that $\tilde{d}$ is compatible with the topology of $\mathcal X$.
	Let $$A_1=\max_{0< i\leq
		n_0-1}\left\{\frac{D(f^{-i}x,f^{-i}y)}{k_1^{i}}\right\} \quad \text{and}\quad A_2=\max_{0< i\leq
		n_0-1}\left\{\frac{D(f^{i}x,f^{i}y)}{k_2^{i}}\right\}.$$
	Then
	\begin{align*}
		\tilde{d}(f^{-1}x,f^{-1}y)
		& \geq  \max_{0\leq i\leq
			n_0-1}\left\{\frac{D(f^{-i-1}x,f^{-i-1}y)}{k_1^{i}}\right\}\\
		& \geq  k_1 \max\left\{A_1, \frac{D(f^{-n_0}x,f^{-n_0}y)}{k_1^{n_0}}\right\}
	\end{align*}
	and
	\begin{align*}
		\tilde{d}(fx,fy)
		& \geq  \max_{0\leq i\leq
			n_0-1}\left\{\frac{D(f^{i+1}x,f^{i+1}y)}{k_2^{i}}\right\}\\
		& \geq k_2 \max\left\{A_2, \frac{D(f^{n_0}x,f^{n_0}y)}{k_2^{n_0}}\right\}.
	\end{align*}
	Thus
	\begin{align*}
		h & \doteq
		\max\left\{\frac{\tilde{d}(f^{-1}x,f^{-1}y)}{k_1},\frac{\tilde{d}(fx,fy)}{k_2}\right\}\\
		& \geq  \max\left\{A_1, A_2,\frac{D(f^{-n_0}x,f^{-n_0}y)}{k_1^{n_0}},\frac{D(f^{n_0}x,f^{n_0}y)}{k_2^{n_0}}\right\} \\
		& =  \max\left\{A_1, A_2, 4\max\{\frac{D(f^{-n_0}x,f^{-n_0}y)}{a^{n_0}},\frac{D(f^{n_0}x,f^{n_0}y)}{b^{n_0}}\}\right\}.
	\end{align*}
	By (\ref{equ 2.9}) we notice that
	\begin{align}\label{equ 2.10}
		\tilde{d}(x,y)=\max\{D(x,y),A_{1},A_{2}\}.
	\end{align}
	If $\tilde{d}(x,y)<\frac{1}{4}\min\{\frac{k_1^{-n_0+1}}{a},\frac{k_2^{-n_0+1}}{b}\}$, then
	$$\max_{|i|\leq
		n-1}\{D(f^{i}(x),f^{i}(y))\}<\frac{1}{4}\min\left\{\frac{1}{a},\frac{1}{b}\right\}.$$
	By Lemma \ref{lem 2.10}, we have
	$$h=\max\left\{\frac{\tilde{d}(f^{-1}x,f^{-1}y)}{k_1},\frac{\tilde{d}(fx,fy)}{k_2}\right\}\geq
	D(x,y).$$
	It implies
	$$h\geq \max\{A_{1},A_{2},D(x,y)\}=\tilde{d}(x,y)$$
	by (\ref{equ 2.10}).
	Hence
	$$\max\left\{\frac{\tilde{d}(f^{-1}x,f^{-1}y)}{a-\gamma},\frac{\tilde{d}(fx,fy)}{b-\gamma}\right\}\geq
	\tilde{d}(x,y).$$
	Since $\mathcal X$ is compact, we can find $\epsilon'>0$ such that if $$\tilde{d}(x,y)\geq\frac{1}{4}\min\{\frac{k_1^{-n_0+1}}{a},\frac{k_2^{-n_0+1}}{b}\},$$
	then
	$$\max\left\{\frac{\tilde{d}(f^{-1}x,f^{-1}y)}{a-\gamma},\frac{\tilde{d}(fx,fy)}{b-\gamma}\right\}\geq\epsilon',$$
	which finishes the proof.	
\end{proof}

\begin{corollary}\label{cor 2.11}
	For any $x,y\in \mathcal X$, 
	\begin{align*}
		D(x,y)\leq \rho(x,y)\leq 4D(x,y) \quad \text{and} \quad	\frac{1}{4}\tilde{d}(x,y)\leq \rho(x,y)\leq 4\tilde{d}(x,y).
	\end{align*}
	Moreover, $f$ and $f^{-1}$ are Lipschitz with respect to $\tilde{d}$.
\end{corollary}

\begin{proof}
	It is straightforward to verify that
	\begin{align*}
		\rho(fx,fy)=&\max \{a^{-n^-(fx,fy)},b^{-n^+(fx,fy)}\}\leq b\rho(x,y), \\
		\rho(f^{-1}x,f^{-1}y)=&\max \{a^{-n^-(f^{-1}x,f^{-1}y)},b^{-n^+(f^{-1}x,f^{-1}y)}\}\leq a\rho(x,y).
	\end{align*}
	From equalities (\ref{equ 2.6}) and (\ref{equ 2.6}), observing that
	$\frac{a}{k_1}=4^{\frac{1}{n_0}}$ and $\frac{b}{k_2}=4^{\frac{1}{n_0}}$,
	one has
	\begin{align*}
		\frac{1}{4}\rho(x,y)\leq D(x,y)\leq\tilde{d}(x,y)=&\max_{0\leq i\leq
			n_0-1}\left\{\frac{D(f^{-i}x,f^{-i}y)}{k_1^{i}},\frac{D(f^{i}x,f^{i}y)}{k_2^{i}}\right\}\\
		\leq &\max \left\{ (\frac{a}{k_1})^{n_0-1}\rho(x,y),(\frac{b}{k_2})^{n_0-1}\rho(x,y) \right\}\\
		\leq &4\rho(x,y),
	\end{align*}
	which implies that \begin{align*}
		\tilde{d}(fx,fy)\leq 16b\tilde{d}(x,y) \quad \text{and} \quad
		\tilde{d}(f^{-1}x,f^{-1}y)\leq 16a\tilde{d}(x,y).
	\end{align*}
	Therefore, $f$ is Lipschitz with
	Lipschitz constant $16b$ for  $\tilde{d}$ and
	$f^{-1}$ is Lipschitz with
	Lipschitz constant $16a$ for  $\tilde{d}$.
\end{proof}

\section{Dimension of invariant sets and measures}\label{Section 3}
In this section, we investigate the pointwise dimension of   invariant measures of  $f$ . We prove that every ergodic invariant measure of  $f$  is exact dimensional, and  compute the Hausdorff dimension of  $\mathcal X$  with respect to the hyperbolic metric  $\tilde{d}$.

\subsection{Relationship between (Bowen) balls and cylinder sets}
For every $x\in \mathcal X$, and any two positive integers $p,q$, we define the
cylinder set as
\begin{align*}
	C_{-q}^p(x)=\{y\in \mathcal X: n^+(x,y)\geq  p,n^-(x,y)\geq   q\}.
\end{align*}
For small enough $r>0$, let $p(r),q(r)$ be two unique positive integers such that
\begin{align}\label{equ 3.11}
	b^{-p(r)} < r\leq
	b^{-p(r)+1} \quad \text{and}\quad a^{-q(r)}< r\leq  a^{-q(r)+1}.
\end{align}
Then we can use the cylinder set to describe the
open ball of $\mathcal X $.

\begin{lemma}\label{lem 3.1}For every $x\in \mathcal X$,
	\begin{itemize}
		\item[\rm (1)] $B_{\rho}(x,r)=C_{-q(r)}^{p(r)}(x).$
		
		\item[\rm (2)] $\lim\limits_{r\to 0}\frac{p(r)+q(r)}{\log \frac{1}{r}}=
		\frac{1}{\log a}+\frac{1}{\log b}.$
	\end{itemize}
\end{lemma}

\begin{proof}
	(1) For any $y\in C_{-q(r)}^{p(r)}(x)$, if $n^+(x,y)\geq   p(r)$ and
	$n^-(x,y)\geq   q(r)$, then $$\rho(x,y)= \max
	\{a^{-n^-(x,y)},b^{-n^+(x,y)}\}<  r.$$ Thus $$C_{-q(r)}^{p(r)}(x)\subseteq
	B_{\rho}(x,r).$$
	For the converse relation, if $y\in B_{\rho}(x,r)$,  that is,
	$$\rho(x,y)= \max
	\{a^{-n^-(x,y)},b^{-n^+(x,y)}\}< r<\min\{a^{-q(r)+1},b^{-p(r)+1}\},$$
	then we have
	$$n^+(x,y)>  p(r)-1 \quad\hbox{ and}\quad n^-(x,y)>  q(r)-1.$$
	Therefore,
	$$B_{\rho}(x,r)\subseteq C_{-q(r)}^{p(r)}(x).$$
	
	(2) By the definitions of
	$p(r),q(r)$, we have
	$$p(r)-1\leq \frac{\log \frac{1}{r}}{\log b}< p(r) \quad \text{and} \quad q(r)-1\leq \frac{\log \frac{1}{r}}{\log a}< q(r).$$
	Thus $$\lim_{r \to 0}\frac{p(r)}{\log \frac{1}{r}}=\frac{1}{\log
		b}\quad\hbox{and}\quad \lim_{r \to 0}\frac{q(r)}{\log \frac{1}{r}}= \frac{1}{\log a}.$$ Hence
	\begin{align*}
		\lim_{r \to 0}\frac{p(r)+q(r)}{\log \frac{1}{r}}=\frac{1}{\log a}+\frac{1}{\log b}.
	\end{align*}
\end{proof}

Recall that the two-sided Bowen ball of radius $r$ and center $x$ with respect to the function $\rho$ is denoted as
$$B_{\rho}(x,-n,m,r)=\{y\in \mathcal X: \rho(f^ix,f^iy)<  r,-n\leq i\leq m,i\in \mathbb Z
\}.$$
For any $r\in (0,1)$,  the $d$-distance  between  points in the Bowen ball $B_{-n}^m(x,\rho,r)$  and $x$  is  not more than the expansive constant $\epsilon$. Namely, if $$\rho(f^{i}x,f^{i}y)=\max\{a^{-n^-(f^{i}x,f^{i}y)},b^{-n^+(f^{i}x,f^{i}y)}\}< r<1, -n\leq i\leq m,$$
then $\min\{n^-(f^{i}x,f^{i}y),n^+(f^{i}x,f^{i}y)\}\geq 1$. Hence $d(f^{i}x,f^{i}y)\leq \epsilon$.

Similarly, we can use the cylinder set to describe the Bowen ball.
\begin{lemma}\label{lem 3.2}
	For every $x\in \mathcal X$, $n,m\in \mathbb N$ and small enough $r\in (0,1)$,
	we have
	$$B_{\rho}(x,-n,m,r)=C_{-q(r)-n}^{p(r)+m}(x).$$
\end{lemma}

\begin{proof}
	Take $y\in B_{\rho}(x,-n,m,r)$.	 Then we have $\rho(f^{i}x,f^{i}y)<r$  for $i=-n,-n+1,\dots ,0,1,\dots m$ and thus by  (\ref{equ 3.11}),
	$$n^+(f^{i}x,f^{i}y)\geq
	p(r)\quad\hbox{ and}\quad n^-(f^{i}x,f^{i}y)\geq  q(r).$$
	By a direct observation,  we get
	\begin{equation}\label{equ 3.12}
		\begin{split}
			n^+(f^{i-1}x,f^{i-1}y)&=n^+(f^{i}x,f^{i}y)+1,\\
			n^-(f^{j+1}x,f^{j+1}y)&=n^-(f^{j}x,f^{j}y)+1
		\end{split}
	\end{equation}
	for $i=1,\dots
	m-1,m,$ and
	$j=-1,\dots,-n+1, -n$. Thus $$n^+(x,y)\geq  p(r)+m \quad\hbox{and}\quad
	n^-(x,y)\geq  q(r)+n.$$ Hence, $B_{\rho}(x,-n,m,r)\subseteq
	C_{-q(r)-n}^{p(r)+m}(x).$
	
	For the converse argument, if $y\in C_{-q(r)-n}^{p(r)+m}(x)$, then
	$$
	n^+(x,y)\geq
	p(r)+m\quad\hbox{ and}\quad n^-(x,y)\geq q(r)+n.$$	We can deduce that
	$$n^+(f^{i}x,f^{i}y)\geq
	p(r)
	\quad\hbox{and}\quad n^-(f^{i}x,f^{i}y)\geq q(r)$$
	for any $-n\leq i\leq m$ by (\ref{equ 3.12}). Therefore, $\rho(f^{i}x,f^{i}y)< r$,  which completes the proof.
\end{proof}

Then we can link the open ball and the Bowen ball by the cylinder set.
\begin{lemma}\label{lem 3.3}
	For every $r_1\in (0,1)$ and small enough $r\in (0,r_1)$,
	there are two positive integers $m,n$ such that 	
	\begin{enumerate}
		\item $m\rightarrow\infty$, $n\rightarrow\infty$, as $r\rightarrow0;$
		
		\item for every $x\in \mathcal X$, $B_{\rho}(x,r)=B_{\rho}(x,-n,m,r_1)=C_{-q(r_1)-n}^{p(r_1)+m}(x);$
		
		\item $\lim\limits_{r\to 0}\frac{m+n}{\log \frac{1}{r}}= \frac{1}{\log a}+\frac{1}{\log b}.$
	\end{enumerate}
\end{lemma}

\begin{proof}
	Let $m=p(r)-p(r_1)$ and $n=q(r)-q(r_1)$.
	Then the desired results  follow directly from Lemmas \ref{lem 3.1} and \ref{lem 3.2}.
\end{proof}

\subsection{Proof of Theorem \ref{thm 1.2}}
Since $f\colon (\mathcal X,\tilde{d})\to (\mathcal X,\tilde{d})$ is a homeomorphism, then one can show that we can use two-sided Bowen ball to describe entropies.
\begin{proposition}\label{prop 3.4}
The Brin-Katok local entropy of $f$ at $x$ with respect to $\mu \in \mathcal{M}(f)$ is equivalent with
\begin{equation*}
	\begin{split}
		h_{\mu}(f,x) & =\lim_{r \to 0} \limsup_{m,n \to \infty}-\frac{\log\mu(B_{\tilde{d}} (x,-n,m,r))}{n+m}
		 = \lim_{r \to 0} \liminf_{m,n \to \infty}-\frac{\log\mu(B_{\tilde{d}} (x,-n,m,r))}{n+m}.
	\end{split}
\end{equation*}
and the topological entropy of $f$ is equivalent with
\begin{equation*}
	\begin{split}h_{top}(f)=\lim_{r\to 0}\limsup_{m,n \to \infty}\frac{1}{m+n}\log r_{d}(f,-n,m,r)
=\lim_{r\to 0}\liminf_{m,n \to \infty}\frac{1}{m+n}\log r_{d}(f,-n,m,r),
	\end{split}
\end{equation*}
where the corresponding notions are defined in Definition \ref{def 2.1}.
\end{proposition}

\begin{proof}
The first result  was shown in \cite{BK83,You82}. We show the equations for topological entropy.  Fix $r$ sufficiently small. Let $E$ be an $(n+m,r)$spanning set of $\mathcal X$ with the 
smallest cardinality. That is $\mathcal X=\bigcup_{x\in E}B_d(x,n+m,r)$. Since $f$ is a homeomorphism, then
$B_d(x,n+m,r)=f^{-n}B_d(f^nx,-n,m,r)$. Notice that $E$ is a finite set. One has
\begin{align*}
\mathcal X=\bigcup_{x\in E}B_d(f^nx,-n,m,r).
\end{align*}
Set $F=f^n E$. Then $r_d(f,-n,m,r)\leq r_d(f,n+m,r)$. Moreover, by considering negative orbits, we have $r_d(f,-n,m,r)\leq r_d(f^{-1},n+m,r)$. Combining with the fact that $h_{top}(f)=h_{top}(f^{-1})$, we finish the proof.
\end{proof}

By invoking Corollary \ref{cor 2.11} and Proposition \ref{prop 3.4},  one can readily confirm that the  function $\rho$ and metrics $D,\tilde{d}$ share the same topological entropy and Brin-Katok local entropy.
\begin{theorem}\label{thm 3.5}
	Let $f$ be an expansive homeomorphism of $\mathcal X$ and $\tilde{d}$ be the metric  obtained in Theorem \ref{thm 1.1}. For every $\mu\in \mathcal M(f)$, one has
	$$d_{\mu}(x,\tilde{d})=h_{\mu}(f,x)\left(\frac{1}{\log a}+\frac{1}{\log b}\right)$$
	for	$\mu$-a.e. $x\in \mathcal X$.
	If $\mu\in{\mathcal E}(f)$, then
	$$\dim_H(\mu,\tilde{d})=\dim_B(\mu,\tilde{d})=h_{\mu}(f)\left(\frac{1}{\log a}+\frac{1}{\log b}\right).$$

\end{theorem}

\begin{proof}
	Let $\mu\in \mathcal M(f)$.	We first show that
	for $\mu$-a.e. $x\in \mathcal X$,
	\begin{align*}
		\overline{d}_{\mu}(x,\tilde{d})
		\leq  h_\mu(f,x)\left(\frac{1}{\log a}+\frac{1}{\log b}\right).
	\end{align*}
	For every $x\in \mathcal X$ and $r>0$, by Corollary \ref{cor 2.11}, we notice that
	\begin{align}\label{equ 3.13}
		B_{\tilde{d}}(x,\frac{1}{4}r)\subseteq B_{\rho}(x,r)\subseteq B_{\tilde{d}}(x,4r)
	\end{align}
	By the definition of pointwise dimension, we have
	$$\limsup_{r\to 0}\frac{\log\mu(B_{\tilde{d}}(x,r))}{\log r}=\limsup_{r\to 0}\frac{\log\mu(B_{\rho}(x,r))}{\log r},$$
	which is also valid if we replace $\limsup$ by $\liminf$.
	Fix  $r_{1}\in (0,1)$. For every small enough $r\in (0,r_1)$, by Lemma \ref{lem 3.3}, we obtain that
	\begin{align*}
		\limsup_{r \to 0}\frac{\log\mu(B_{\rho}(x,r))}{\log r}
		& =\limsup_{r \to 0} \frac{\log\mu(B_{\rho}(x,-n,m,r_1))}{m+n}\frac{m+n}
		{\log r}\\
		& =  \limsup_{m,n \to \infty}\frac{\log\mu(B_{\rho}(x,-n,m,r_1))}{-(m+n)}\lim_{r \to 0}\frac{m+n}
		{\log \frac{1}{r}}\\
		& = \left(\frac{1}{\log a}+\frac{1}{\log b}\right)\limsup_{m,n \to \infty}\frac{\log\mu(B_{\rho}(x,-n,m,r_1))}{-(m+n)}\\
		& \leq (\frac{1}{\log a}+\frac{1}{\log
			b})\limsup_{m,n \to \infty}\frac{\log\mu(B_{\tilde{d}}(x,-n,m,\frac{1}{4}r_1))}{-(m+n)}\\
		&\leq h_{\mu}(f,x)(\frac{1}{\log a}+\frac{1}{\log
			b}).
	\end{align*}
	
	To derive the converse inequality, we have
	\begin{align*}
		\liminf_{r \to 0}\frac{\log\mu(B_{\rho}(x,r))}{\log r}
		& =\liminf_{r \to 0} \frac{\log\mu(B_{\rho}(x,-n,m,r_1))}{-(m+n)}\frac{m+n}
		{\log \frac{1}{r}}\\
		& =  \liminf_{m,n \to \infty}\frac{\log\mu(B_{\rho}(x,-n,m,r_1))}{-(m+n)}\lim_{r \to 0}\frac{m+n}
		{\log \frac{1}{r}}\\
		& =  \left(\frac{1}{\log a}+\frac{1}{\log b}\right)\liminf_{m,n \to \infty}\frac{\log\mu(B_{\rho}(x,-n,m,r_1))}{-(m+n)}\\
		& \geq  (\frac{1}{\log a}+\frac{1}{\log
			b})\liminf_{m,n \to \infty}\frac{\log\mu(B_{\tilde{d}}(x,-n,m,4r_1))}{-(m+n)}.
	\end{align*}
	From the arbitrariness of $r_1$  and for $\mu$-a.e. $x\in \mathcal X$, we get that
	\begin{align*}
		\underline{d}_{\mu}(x,\tilde{d})\geq h_\mu(f,x)   \left(\frac{1}{\log a}+\frac{1}{\log b}\right).
	\end{align*}
	Hence, for $\mu$-a.e. $x\in \mathcal X$, we have	$$d_{\mu}(x,\tilde{d})=   h_\mu(f,x)\left(\frac{1}{\log a}+\frac{1}{\log b}\right).$$
	
	If $\mu \in \mathcal E(f)$, noticing that $h_{\mu}(f)=h_\mu(f,x), \mu\text{-a.e.~}x\in \mathcal X$, and	
	combining with Lemma \ref{lem 2.3}, we obtain
	$$d_{\mu}(x,\tilde{d})=h_{\mu}(f)\left(\frac{1}{\log a}+\frac{1}{\log b}\right)=\dim_H(\mu,\tilde{d})=\dim_B(\mu,\tilde{d})$$ for $\mu$-a.e. $x\in \mathcal X.$ The proof is thus complete.
\end{proof}

Furthermore, we can describe the precise relation between topological entropy and the Hausdorff and Box dimension.
\begin{theorem}\label{thm 3.6}
	Let $f$ be an expansive homeomorphism of $\mathcal X$ and $\tilde{d}$ be the metric  obtained in Theorem \ref{thm 1.1}. Then we have
	$$\dim_H(\mathcal X,\tilde{d})=\dim_B(\mathcal X,\tilde{d})=h_{top}(f)\left(\frac{1}{\log a}+\frac{1}{\log b}\right).$$
\end{theorem}

\begin{proof}
	By Theorem \ref{thm 3.5}, we have
	\begin{align*}
		\dim_H(\mathcal X,\tilde{d}) \geq  \sup_{\mu\in {\mathcal E}(f)} \dim_H(\mu,\tilde{d})& \geq  \sup_{\mu\in {\mathcal E}(f)} \left\{h_{\mu}(f)\left(\frac{1}{\log a}+\frac{1}{\log b}\right)\right\}\\
		& =h_{top}(f)\left(\frac{1}{\log a}+\frac{1}{\log b}\right).
	\end{align*}
	Similarly, since $\frac{1}{4}\tilde{d}(x,y)\leq \rho(x,y)\leq 4\tilde{d}(x,y)$,
	we notice that the  function $\rho$ and the metric $\tilde{d}$ do not have any effect on the computation of dimension  and topological entropy.
	
	Let $N(f,\rho,r)$ denote the minimum number of open balls with radius $r$ in terms of the function $\rho$
	that
	covers $\mathcal X$.
	For sufficiently small $0<r<r_1<1$, by Lemma \ref{lem 3.3}, there exist two positive integers $m,n$ such that
	\begin{itemize}
		\item[(1)]$m\rightarrow\infty$, $n\rightarrow\infty$, as $r\rightarrow0$;
		
		\item[(2)] for any $x\in \mathcal X$, $B_{\rho}(x,r)=B_{\rho}(x,-n,m,r_1)$.
	\end{itemize}
	
	Let $E$ be an  $(n+m,r_1)$-spanning set of $\mathcal{X}$ with
	$r_{\tilde{d}}(f,n+m,r_1)=|E|$. Since $f$ is a homeomorphism, if taking images by $f^{-n}$, 
	we get  $N(\mathcal X,\rho,4r)\leq  r_{m+n}(f,\tilde{d},r_1)$ by Proposition \ref{prop 3.4}.
	It follows that
	$$\lim_{r\rightarrow 0}\limsup_{n,m\to \infty}\frac{\log N(\mathcal X,\rho,4r)}{m+n} \leq h_{top}(f).$$
	By Lemma \ref{lem 3.3} again, we have
	\begin{align*}
		\overline{\dim}_B(\mathcal X,\tilde{d})& = \limsup_{r\to 0}\frac{\log N(\mathcal X,\rho,4r)}{\log
			\frac{1}{4r}}\\
		& \leq  \limsup_{m,n\to \infty}\frac{\log r_{m+n}(f,\tilde{d},r_1)}{m+n}\lim_{r\to 0}\frac{m+n}{\log \frac{1}{4r}}\\
		& \leq  h_{top}(f)\left(\frac{1}{\log a}+\frac{1}{\log b}\right).
	\end{align*}
	Combining with Lemma \ref{lem 2.3}, we have
	$$\dim_{H}(\mathcal X,\tilde{d})=\overline{\dim}_B(\mathcal X,\tilde{d})=\underline{\dim}_B(\mathcal X,\tilde{d})=h_{top}(f)\left(\frac{1}{\log a}+\frac{1}{\log b}\right),$$
	which finishes the proof.
\end{proof}

\begin{proof}[Proof of Theorem \ref{thm 1.2}]
	It directly follows  by  Theorems \ref{thm 3.5} and \ref{thm 3.6}.
\end{proof}

\section{Variational principles for $r$-neutralized entropy}\label{Section 4}
In this section, we compute the $r$-neutralized entropy of an expansive homeomorphism, and establish  variational principles in terms of the hyperbolic metric $\tilde{d}$.

\subsection{Several notations of $r$-neutralized entropies}
Through the section, let $\tilde{d}$ be the metric obtained in Theorem \ref{thm 1.1}. Contrary to the classical definition of entropy, in our investigation of neutralized entropy, we concentrate on its dependence on $r$.

\begin{enumerate}
	\item \emph{Upper and lower $r$-neutralized topological entropy}:
	\begin{align*}
		\overline{h}_{top,\tilde{d}	}^{r}(f):=&\limsup_{m,n\to\infty}\frac{1}{m+n}\log r_{\tilde{d}}(f,-n,m,e^{-(m+n)r}),\\
		\underline{h}_{top,\tilde{d}}^{r}(f):=&\liminf_{m,n\to\infty}\frac{1}{m+n}\log r_{\tilde{d}}(f,-n,m,e^{-(m+n)r}),
	\end{align*}
	where  $r_{\tilde{d}}(f,-n,m,e^{-(m+n)r})=\min \left\{\# E\colon \mathcal X=\bigcup_{x\in E} B_{\tilde{d}}(x,-n,m,e^{-(m+n)}r)\right\}.$
	
	\item \emph{Upper and lower $r$-neutralized Katok entropy of $\mu\in \mathcal M(f)$ and $\delta\in (0,1)$}:
	\begin{align*}
		\overline{h}_{\mu,\tilde{d}}^{K,r}(f,\delta):=&\limsup_{m,n\to\infty}\frac{1}{m+n}\log r_{\tilde{d}}(f,\mu,\delta,-n,m,e^{-(m+n)r}),\\
		\underline{h}_{\mu,\tilde{d}}^{K,r}(f,\delta):=&\liminf_{m,n\to\infty}\frac{1}{m+n}\log r_{\tilde{d}}(f,\mu,\delta,-n,m,e^{-(m+n)r}),\end{align*}
	where  
	\begin{align*}
		r_{\tilde{d}}(f,\mu,-n,m,e^{-(m+n)r},\delta)=\min \left\{\# K\colon \mu(\bigcup_{x\in K} B_{-n}^{m}(x,\tilde{d},e^{-(m+n)r})  )\geq 1-\delta\right\}.
	\end{align*}
	
	\item \emph{Upper and lower $r$-neutralized Brin-Katok local entropy of $\mu\in \mathcal M(f)$}:
	\begin{align*}
		\overline{h}_{\mu,\tilde{d}}^{BK,r}(f):=&\int \overline{h}_{\mu,\tilde{d}}^{BK,r}(f,x) d\mu(x)\\
		=&\int\limsup_{m,n\to\infty}-\frac{1}{m+n}\log \mu \left(	B_{\tilde{d}}(x,-n,m,e^{-(m+n)}r)\right)d\mu(x),\\  \underline{h}_{\mu,\tilde{d}}^{BK,r}(f):=&\int \underline{h}_{\mu,\tilde{d}}^{BK,r}(f,x)d\mu(x)\\
		=&\int\liminf_{m,n\to\infty}-\frac{1}{m+n}\log \mu \left(	B_{\tilde{d}}(x,-n,m,e^{-(m+n)}r)\right)d\mu(x).
	\end{align*}	
\end{enumerate}
Simmilar with \cite[(1)]{OR24}, one can show that for $\mu$-a.e. $x\in \mathcal X$, $\overline{h}_{\mu,\tilde{d}}^{BK,r}(f,x)$ and $\underline{h}_{\mu,\tilde{d}}^{BK,r}(f,x)$ are $f$-invariant. Moreover, if $\mu$ is ergodic, then
$\overline{h}_{\mu,\tilde{d}}^{BK,r}(f,x)=\overline{h}_{\mu,\tilde{d}}^{BK,r}(f)$ and $\underline{h}_{\mu,\tilde{d}}^{BK,r}(f,x)=\underline{h}_{\mu,\tilde{d}}^{BK,r}(f)$ for $\mu$-a.e. $x\in \mathcal X$.

If the  upper and the  lower $r$-neutralized topological entropy (or Brin-Katok, Katok entropy) are equal, we denote their common value by $h_{top,\tilde{d}}^{r}(f)$ (or $h_{\mu,\tilde{d}}^{BK,r}(f)$,$h_{\mu,\tilde{d}}^{K,r}(f,\delta)$), respectively.

\begin{remark}
	Different from classic Brin-Katok local entropy, Katok entropy and topological entropy, the $r$-neutralized entropy quantities are dependent of the metric compatible with topology on $\mathcal{X}$. See \cite[Proposition A.1]{DQ25} or Section \ref{Section 6} for the dependence of $r$. 
\end{remark}

We introduce a direct comparison of the three entropy quantities.
\begin{lemma}[\cite{DQ25}]\label{lem 4.2}
	Let	$(\mathcal X,f)$ be a TDS and $\mu \in \mathcal E (f)$. Then for any $r>0$ and $\delta\in (0,1)$,
	\begin{align*}
		\underline{h}_{top,\tilde{d}}^{r}(f)\geq \underline{h}_{\mu,\tilde{d}}^{K,r}(f,\delta)\geq \underline{h}_{\mu,\tilde{d}}^{BK,r}(f).
	\end{align*}
	In addition, the first inequality holds for any $\mu\in \mathcal M(f)$.
\end{lemma}

Given $m,n\in \mathbb N$ and $r>0$.
The proof of the subsequent result follows a similar argument to  Lemma \ref{lem 3.2}, since $e^{-(n+m)r} \in (0,1)$.

\begin{lemma}\label{lem 4.3}
	For every $x\in \mathcal X$, $r>0$ and for any positive integers $m,n$, we have
	$$B_{\rho}(x,-n,m,e^{-(n+m)r})= C_{-q(e^{-(n+m)r})-n}^{p(e^{-(n+m)r})+m}(x).$$
\end{lemma}

Using this conclusion, we can state the following result.
\begin{lemma}\label{lem 4.4}
	Let   $r_1\in (0,1)$ and $r_{2}\in (0,\frac{3}{\frac{1}{\log a}+\frac{1}{\log b}})$. Then  for any small enough $r<\min\{r_1,e^{-2r_2}\}$, there exist four positive integers $m_{1}:=m_{1}(r,r_{1}),n_{1}:=n_{1}(r,r_{1})$ and  $m_{2}:=m_{2}(r,r_{2}),n_{2}:=n_{2}(r,r_{2})$  satisfying the following conditions: for any $x\in \mathcal X$,
	\begin{enumerate}
		\item $B_{\rho}(x,-(n_2+2),m_2,e^{-(n_2+m_2+2)r_2})
		\subseteq B_{\rho}(x,r)
		\subseteq B_{\rho}(x,-(n_2-2),m_2,e^{-(n_2+m_2-2)r_2})$;
		
		\item  $
		B_{\rho}(x,r)=B_{\rho}(x,-n_1,m_1,r_{1});$

		\item $n_1,n_2,m_1,m_2\rightarrow\infty$, as $r\rightarrow0$.
	\end{enumerate}
	Moreover,
	\begin{equation*}
		\begin{split}
			\lim_{r \to 0}\frac{m_1+n_1}{-\log r}=\frac{1}{\log a}+\frac{1}{\log b},\quad 	\lim_{r \to 0}\frac{m_2+n_2}{-\log r}=\frac{\frac{1}{\log a}+\frac{1}{\log b}}{1+\frac{r_2}{\log a}+\frac{r_2}{\log b}} .
		\end{split}
	\end{equation*}
\end{lemma}

\begin{proof}
	For the sake of simplicity, we set $k=\frac{1}{\log a}+\frac{1}{\log b}$.
	
	(1) 
	The existence of $m_1,n_1$ is directly from Lemma \ref{lem 3.3}.
	We only need to show the existence of $m_2,n_2$. A key point is whether  the  equations
	\begin{align}\label{equ 4.14}
		\left\{\begin{array}{cc}
			m=p(r)-p(e^{-(m+n)r_{2}})\\
			n=q(r)-q(e^{-(m+n)r_{2}})
		\end{array}
		\right.
	\end{align}
	have positive integer solutions $m_2,n_2$.
	Put $h=m+n$.
	By definition we have
	\begin{align*}
		p(e^{-hr_{2}})-1\leq \frac{hr_{2}}{\log b}< p(e^{-hr_{2}}) \quad \text{and} \quad  p(r)-1\leq \frac{\log\frac{1}{r}}{\log b}<p(r)
	\end{align*}
	and
	\begin{align*}
		q(e^{-hr_{2}})-1\leq \frac{hr_{2}}{\log a}< q(e^{-hr_{2}}) \quad \text{and} \quad  q(r)-1\leq \frac{\log\frac{1}{r}}{\log a}< q(r).
	\end{align*}
	This gives that
	\begin{align}\label{equ 4.135}
		-2+k\log \frac{1}{r}-hr_2k
		<
		p(r)-p(e^{-hr_{2}})+q(r)-q(e^{-hr_{2}})
		<2+k\log \frac{1}{r}-hr_2k.
	\end{align}
	Consider the solution of the inequality
	\begin{align}\label{equ 4.16}
		-2+k\log \frac{1}{r}-kr_2h<h<2+k\log \frac{1}{r}-kr_2h.
	\end{align}
	We obtain that $h\in \left(\frac{k\log \frac{1}{r}-2}{1+kr_2},\frac{k\log \frac{1}{r}+2}{1+kr_2}\right)$.
	In light of $r_2\in (0,\frac{3}{k})$ and $r$ is sufficiently small, then
	the equation \eqref{equ 4.16} has a positive integer solution $h$. Moreover, since
	$0<r<\min\{r_1,e^{-2r_2}\}$, then it implies that $r<e^{-hr_{2}}.$
	
	Fix such $h$ in \eqref{equ 4.16} and $r,r_2$. Based on (\ref{equ 4.135}),  we notice that
	$p(r)-p(e^{-hr_{2}})+q(r)-p(e^{-hr_{2}})$ can only take up to $3$ integers and the absolute value of their difference from $h$ does not exceed $2$. Namely,
	\begin{align}\label{equ 4.137}
		h-2\leq p(r)-p(e^{-hr_{2}})+q(r)-q(e^{-hr_{2}})\leq  h+2.
	\end{align}
	For each integer $h$  satisfying (\ref{equ 4.16}), there exists an integer $j\in [-2,2]$ such that $$p(r)-p(e^{-hr_{2}})+q(r)-q(e^{-hr_{2}})=h+j.$$
	Take $m_2=p(r)-p(e^{-hr_{2}})$ and $n_2+j=q(r)-q(e^{-hr_{2}})$. Then $m_2,n_2+j\in \mathbb N$ since $r<e^{-hr_{2}}$.
	This suggests that for any given $0<r_{2}<\frac{3}{k}$,
	there exists  an unique integer $j\in [-2,2]$ such that
	the  equations
	\begin{align}\label{equ 4.138}
		\left\{\begin{array}{cc}
			m=p(r)-p(e^{-(m+n)r_{2}})\\
			n+j=q(r)-q(e^{-(m+n)r_{2}})
		\end{array}
		\right.
	\end{align}
	have positive integer solutions $m_2,n_2$.
	
	Since $h\in \left(\frac{k\log \frac{1}{r}-2}{1+kr_2},\frac{k\log \frac{1}{r}+2}{1+kr_2}\right)$ and $k,r_2$ are  constants independent of $r$ and $c \to \infty$ as $r\to 0$, we have $h\to\infty$  as $r\to 0$. It indicates that $m_2,n_2 \to \infty$  as $r\to 0$.
	By taking $r$ sufficiently small, we can ensure that  $m_2,n_2$ are positive integers.

	By virtue of $B_{\rho}(x,r)=C^{p(r)}_{-q(r)}(x)$ and
	\begin{align*}
		C_{-(q(e^{-(n_2+m_2+2)r_2})+n_2+2)}^{p(e^{-(n_2+m_2+2)r_2})+m_2}(x)
		\subseteq& C^{p(r)}_{-q(r)}(x)=C^{m_2+p(e^{-(m_2+n_2)r_{2}})}_{-n-j-q(e^{-(m_2+n_2)r_{2}})}(x)\\
		\subseteq& C_{-(q(e^{-(n_2+m_2-2)r_2})+n_2-2)}^{p(e^{-(n_2+m_2-2)r_2})+m_2}(x).
	\end{align*}
	for any $j\in [-2,2]$,
	using Lemma \ref{lem 4.3},
	we obtain that
	$$B_{\rho}(x,-(n_2+2),m_2,e^{-(n_2+m_2+2)r_2})\subseteq B_{\rho}(x,r)\subseteq B_{\rho}(x,-(n_2-2),m_2,e^{-(n_2+m_2-2)r_2}).$$
	This completes the proof of (1).
	
	(2) and (3) are  direct results of (\ref{equ 4.138}) and Lemma \ref{lem 3.3}.

	Since $m_2,n_2 \to \infty$ as $r \to 0$, we have
	\begin{align*}
		\lim_{r \to 0}\frac{p(r)+q(r)-m_2-n_2}{m_2+n_2}=\lim_{m_2,n_2 \to \infty}\frac{p(e^{-(n_2+m_2)r_{2}})+q(e^{-(n_2+m_2)r_{2}})}{m_2+n_2}
		=r_{2}k.
	\end{align*}
	Hence,
	$$\lim_{r \to 0}\frac{m_2+n_2}{-\log r}=\lim_{r \to 0}\frac{m_2+n_2}{p(r)+q(r)} \lim_{r \to 0}\frac{p(r)+q(r)}{-\log r}=\frac{k}{1+r_{2}k},$$
	which completes the proof.
\end{proof}

\subsection{Proof of Theorem \ref{thm 1.3}}
By Corollary \ref{cor 2.11}, we can also  approximate $r$-neutralized  entropy  with respect to the hyperbolic metric $\tilde{d}$  by using the function $\rho$. The following result elucidates the relationship between $r$-neutralized entropy and dimensions, bearing a resemblance to the forms presented in \cite{DQ25,You82} for diffeomorphisms..
\begin{proposition}\label{prop 4.4}
	Let $f$ be an expansive homeomorphism of $\mathcal X$. Given any $0<r_2<\frac{3}{\frac{1}{\log a}+\frac{1}{\log b}}$, then for every  $\mu \in \mathcal M(f)$ and  $\mu$-a.e. $x\in \mathcal X,$
	\begin{align*}
		d_{\mu}(x,\tilde{d})=\overline{h}_{\mu,\tilde{d}}^{BK,r_2}(f,x)\left(\frac{1}{r+\frac{1}{\frac{1}{\log a}+\frac{1}{\log b}}}\right)
		=\underline{h}_{\mu,\tilde{d}}^{BK,r_2}(f,x)\left(\frac{1}{r+\frac{1}{\frac{1}{\log a}+\frac{1}{\log b}}}\right).
	\end{align*}
	If $\mu\in \mathcal E(f)$, then
	\begin{align*}
		\dim_{H}(\mu,\tilde{d})=\overline{h}_{\mu,\tilde{d}}^{BK,r_2}(f)\left(\frac{1}{r+\frac{1}{\frac{1}{\log a}+\frac{1}{\log b}}}\right)
		=\underline{h}_{\mu,\tilde{d}}^{BK,r_2}(f)\left(\frac{1}{r+\frac{1}{\frac{1}{\log a}+\frac{1}{\log b}}}\right).
	\end{align*}
\end{proposition}
\begin{proof}
	
	Notice that for $\mu$-a.e. $x\in \mathcal X,$
	\begin{align}\label{equ 4.19}	
		d_{\mu}(x,\tilde{d})=h_{\mu}(f,x)\left(\frac{1}{\log a}+\frac{1}{\log b}\right)
	\end{align}	
	by Theorem \ref{thm 3.5}.
	Given small enough $r_1>0$, by Lemma \ref{lem 4.3},  there exist four positive integers $m_1,n_1,m_2,n_2$ such that for any $x\in \mathcal X,$
	\begin{align*}
		B_{\rho}(x,-(n_2+2),m_2,e^{-(n_2+m_2+2)r_2})\subseteq& B_{\rho}(x,-n_1,m_1,r_{1})\\
		\subseteq& B_{\rho}(x,-(n_2-2),m_2,e^{-(n_2+m_2-2)r_2}).
	\end{align*}
	Take $m'_2=m_2+[\frac{\log 2}{r_2}]+1$,
	$n'_2=n_2+[\frac{\log 2}{r_2}]+1$ such that $\frac{1}{4}e^{-(m_2+n_2+2)r_2}\geq e^{-(m_2'+n'_2+2)r_2}$ and
	take $m''_2=m_2-[\frac{\log 2}{r_2}]$,
	$n''_2=n_2-[\frac{\log 2}{r_2}]$ such that $e^{-(m_2''+n_2''-2)r_2}\geq 4e^{-(m_2+n_2-2)r_2}$.
	
Obviously,
	\begin{equation*}
		\begin{split}
			\underline{h}_{\mu,\tilde{d}}^{BK,r_2}(f,x)=& \liminf_{m_2',n_2' \to \infty}\frac{\log\mu(B_{\tilde{d}}(x,-(n_2'+2),m_2',e^{-(m_2'+n'_2+2)r_2}))}{-(m'_2+n_2'+2)}\\
			\geq & \liminf_{n_2,m_2 \to \infty}\frac{\log\mu(	B_{\tilde{d}}(x,-(n_2+2),m_2,\frac{1}{4}e^{-(m_2+n_2+2)r_2}))}{-(m_2+n_2+2)}\\
			\geq & \liminf_{n_2,m_2 \to \infty}\frac{\log\mu(	B_{\rho}(x,-(n_2+2),m_2,e^{-(m_2+n_2+2)r_2}))}{-(m_2+n_2+2)}\\
			\geq &  \liminf_{n_1,m_1 \to \infty}\frac{\log\mu(	B_{\rho}(x,-n_1,m_1,r_{1}))}{-(n_1+m_1)}\lim_{r\to 0}\frac{n_1+m_1}{\log \frac{1}{r}}\cdot \frac{\log \frac{1}{r}}{n_2+m_2+2}
			\\
			\geq&  \liminf_{n_1,m_1 \to \infty}\frac{\log\mu(	B_{\tilde{d}}(x,-n_1,m_1,4r_{1}))}{-(n_1+m_1)}\left(1+r_2\left(\frac{1}{\log a}+\frac{1}{\log b}\right)\right)
		\end{split}
	\end{equation*}
	and
	\begin{equation*}
		\begin{split}
			\overline{h}_{\mu,\tilde{d}}^{BK,r_2}(f,x)=& \limsup_{n_2'',m_2'' \to \infty}\frac{\log\mu(B_{\tilde{d}}(x,-(n_2''-2),m_2'',e^{-(m_2''+n_2''-2)r_2}))}{-(n_2''+m_2''-2)}\\
			\leq & \limsup_{n_2,m_2 \to \infty}\frac{\log\mu(B_{\tilde{d}}(x,-(n_2-2),m_2,4e^{-(m_2+n_2-2)r_2}))}{-(n_2+m_2-2)}\\
			\leq & \limsup_{n_2,m_2 \to \infty}\frac{\log\mu(B_{\rho}(x,-(n_2-2),m_2,e^{-(n_2+m_2-2)r_2}))}{-(n_2+m_2-2)}\\
			\leq &  \limsup_{n_1,m_1 \to \infty}\frac{\log\mu(	B_{\rho}(x,-n_1,m_1,r_{1}))}{-(n_1+m_1)}\lim_{r \to 0}\frac{n_1+m_1+1}{\log \frac{1}{r}}\cdot \frac{\log \frac{1}{r}}{n_2+m_2-2}\\
			\leq & \left(1+r_2\left(\frac{1}{\log a}+\frac{1}{\log b}\right)\right)\limsup_{m_1,n_1 \to \infty}\frac{\log\mu(B_{\tilde{d}}(x,-n_1,m_1,\frac{1}{4}r_{1}))}{-(n_1+m_1)}.
		\end{split}
	\end{equation*}
	Let $r_{1}\to 0$.
	It shows that for $\mu$-a.e. $x\in \mathcal X$,
	\begin{align*}
		\overline{h}_{\mu,\tilde{d}}^{BK,r_2}(f,x)=&
		\underline{h}_{\mu,\tilde{d}}^{BK,r_2}(f,x)\\
		=&h_{\mu}(f,x)\left(1+r_{2}\left(\frac{1}{\log a}+\frac{1}{\log b}\right)\right)\\
		=&d_{\mu}(x,\tilde{d})\left(r_2+\frac{1}{\frac{1}{\log a}+\frac{1}{\log b}}\right)
	\end{align*}
	by (\ref{equ 4.19}).
	If $\mu\in \mathcal E(f)$, for $\mu$-a.e. $x\in \mathcal X$,
	$$\overline{h}_{\mu,\tilde{d}}^{BK,r_2}(f,x)=
	\underline{h}_{\mu,\tilde{d}}^{BK,r_2}(f,x)=\overline{h}_{\mu,\tilde{d}}^{BK,r_2}(f)=
	\underline{h}_{\mu,\tilde{d}}^{BK,r_2}(f)$$
	and
	$$\dim_{H}(\mu,\tilde{d})=h_{\mu,\tilde{d}}^{BK,r_2}(f)\left(\frac{1}{r_2+\frac{1}{\frac{1}{\log a}+\frac{1}{\log b}}}\right).
	$$
\end{proof}

\begin{proposition}\label{prop 4.5}
	Let $\tilde{d}$ be the metric obtained in Theorem \ref{thm 1.1}. Given any $0<r_2<\frac{3}{\frac{1}{\log a}+\frac{1}{\log b}}$,
	then we have
	\begin{align*}
		\dim_{H}(\mathcal X,\tilde{d})=\overline{h}_{top,\tilde{d}}^{r_2}(f)\left(\frac{1}{r_2+\frac{1}{\frac{1}{\log a}+\frac{1}{\log b}}}\right)=\underline{h}_{top,\tilde{d}}^{r_2}(f)
		\left(\frac{1}{r_2+\frac{1}{\frac{1}{\log a}+\frac{1}{\log b}}}\right).	
	\end{align*}
\end{proposition}

\begin{proof}	
	Given small enough $r_1>0$, by Lemma \ref{lem 4.3}, there exist four positive integers $m_1,n_1,m_2,n_2$ such that
	\begin{align*}
		B_{\rho}(x,-(n_2+2),m_2,e^{-(n_2+m_2+2)r_2})\subseteq& B_{\rho}(x,-n_1,m_1,r_{1})\\
		\subseteq& B_{\rho}(x,-(n_2-2),m_2,e^{-(n_2+m_2-2)r_2}).
	\end{align*}
	It implies that
	\begin{align*}
		r_{\rho}(f,-(n_2-2),m_2,e^{-(m_2+n_2-2)r_2})\leq & r_{\rho}(f,-n_1,m_1,r_{1})\\
		\leq& r_{\rho}(f,-(n_2+2),m_2,e^{-(m_2+n_2+2)r_2}).
	\end{align*}
	Take $m'_2=m_2+[\frac{\log 2}{r_2}]+1$ ,
	$n'_2=n_2+[\frac{\log 2}{r_2}]+1$ such that $\frac{1}{4}e^{-(m_2+n_2+2)r_2}\geq e^{-(m_2'+n'_2+2)r_2}$ and
	$m''_2=m_2-[\frac{\log 2}{r_2}]$ ,
	$n''_2=n_2-[\frac{\log 2}{r_2}]$ such that $e^{-(m_2''+n_2''-2)r_2}\geq 4e^{-(m_2+n_2-2)r_2}$. 	
	It indicates that
	\begin{align*}
		\overline{h}_{top,\tilde{d}}^{r_2}(f)=&\limsup_{m_2'',n_2''\to\infty}\frac{\log r_{\tilde{d}}(f,-(n_2''-2),m_2'',e^{-(m_2''+n_2''-2)r_2})}{m_2''+n_2''-2}\\
		\leq
		&\limsup_{m_2,n_2\to\infty}\frac{\log r_{\tilde{d}}(f,-(n_2-2),m_2,4e^{-(m_2+n_2-2)r_2})}{m_2+n_2-2}\\
		\leq
		&\limsup_{m_2,n_2\to\infty}\frac{\log r_{\rho}(f,-(n_2-2),m_2,e^{-(m_2+n_2-2)r_2})}{m_2+n_2-2}\\
		\leq &\limsup_{m_1,n_1\to\infty}\frac{\log r_{\rho}(f,-n_1,m_1,r_{1})}{n_1+m_1}\cdot(1+r_2k)\\
		\leq &(1+r_2\left(\frac{1}{\log a}+\frac{1}{\log b}\right))\limsup_{m_1,n_1\to\infty}\frac{\log r_{\tilde{d}}(f,-n_1,m_1,\frac{1}{4}r_{1})}{n_1+m_1}
	\end{align*}
	by Lemma \ref{lem 4.2}. For the converse inequality,
	\begin{align*}
		\underline{h}_{top,\tilde{d}}^{r_2}(f)=&\liminf_{m'_2,n'_2\to\infty}\frac{\log r_{\tilde{d}}(f,-(n'_2+3),m'_2,e^{-(m'_2+n'_2+3)r_{2}})}{m'_2+n'_2+3}\\
		\geq &\liminf_{m'_2,n'_2\to \infty}\frac{\log r_{\rho}(f,-(n'_2+3),m'_2,\frac{1}{4}e^{-(m_2+n_2+2)r_{2}})}{m'_2+n'_2+3}\\
		\geq &\liminf_{m_2,n_2\to\infty}\frac{\log r_{\rho}(f,-(n_2+2),m_2,e^{-(m_2+n_2+2)r_{2}})}{m_2+n_2+2}\\
		\geq &(1+r_2\left(\frac{1}{\log a}+\frac{1}{\log b}\right))\liminf_{m_1,n_1\to\infty}\frac{\log r_{\rho}(f,-n_1,m_1,r_{1})}{n_1+m_1}\\\
		\geq &(1+r_2\left(\frac{1}{\log a}+\frac{1}{\log b}\right))\liminf_{m_1,n_1\to\infty}\frac{\log r_{\tilde{d}}(f,-n_1,m_1,4r_{1})}{n_1+m_1}.
	\end{align*}
	Letting $r_{1} \to 0$ and combining with Theorem \ref{thm 3.5},
	we have
	$$\overline{h}_{top,\tilde{d}}^{r_2}(f)=
	\underline{h}_{top,\tilde{d}}^{r_2}(f)= h_{top}(f)+r_2\dim_{H}(\mathcal X,\tilde{d})=\left(r_2+\frac{1}{\frac{1}{\log a}+\frac{1}{\log b}}\right)\dim_{H}(\mathcal X,\tilde{d}) .$$
\end{proof}

Based on Propositions \ref{prop 4.4} and \ref{prop 4.5}, it allows us to establish the variational principles in terms of $r$-neutralized Katok entropy and the $r$-neutralized Brin-Katok entropy.
\begin{corollary}[=Corollary \ref{cor 1.4}]
	Let $f$ be an expansive homeomorphism of $\mathcal X$.	Then for any
	$0<r<\frac{3}{\frac{1}{\log a}+\frac{1}{\log b}},$ we have
	\begin{align*}
		h_{top,\tilde{d}}^r(f)=&\left(1+r\left(\frac{1}{\log a}+\frac{1}{\log b}\right)\right)h_{top}(f)\\
		=&\left(1+r\left(\frac{1}{\log a}+\frac{1}{\log b}\right)\right)\sup_{\mu\in {\mathcal M}(f)}h_{\mu}(f)\\
		=&\left(1+r\left(\frac{1}{\log a}+\frac{1}{\log b}\right)\right)\sup_{\mu\in {\mathcal E}(f)}h_{\mu}(f).
	\end{align*}
	Moreover,
	\begin{align*}
		h_{top,\tilde{d}}^r(f)=&\sup\left\{h_{\mu,\tilde{d}}^{BK,r}(f)\colon\mu\in \mathcal M(f)\right\}
		=\sup\left\{h_{\mu,\tilde{d}}^{BK,r}(f)\colon\mu\in \mathcal E(f)\right\}.
	\end{align*}
\end{corollary}

\begin{corollary}\label{cor 4.8}
	Let $f$ be an expansive homeomorphism of $\mathcal X$.	Then for any
	$0<r<\frac{3}{\frac{1}{\log a}+\frac{1}{\log b}}$ and $\mu\in \mathcal E(f)$, the value of $r$-neutralized Katok entropy is independent of the choice of $\delta\in (0,1)$ and
	\begin{align*}
		h_{\mu,\tilde{d}}^{K,r}(f,\delta)
		=h_{\mu}(f)\left(1+r\left(\frac{1}{\log a}+\frac{1}{\log b}\right)\right).
	\end{align*}
	Additionally,	
	\begin{align*}
		h_{top,\tilde{d}}^r(f)
		=\sup\left\{h_{\mu,\tilde{d}}^{K,r}(f,\delta)\colon\mu\in \mathcal E(f)\right\}.
	\end{align*}
\end{corollary}

\begin{proof}[Proof of Theorem \ref{thm 1.3}]
	This is due to
	Propositions \ref{prop 4.4} and \ref{prop 4.5}.
\end{proof}

\begin{remark}
 We partially proved the \cite[Conjecture A.1]{DQ25} for any  hyperbolic metrics of the expansive homeomorphism.
\end{remark}

\section{Variational principles  for $\alpha$-estimation entropy}\label{Section 5}
In this section, we delve into the correlation between the $\alpha$-estimation entropy and the classical entropy within the context of the expansive homeomorphism and the positively expansive map.

\subsection{Variational principles  for expansive homeomorphism}
Let $(\mathcal X,\tilde{d})$ be a compact metric space, where $\tilde{d}$ is obtained in Theorem \ref{thm 1.1}.
For any $\alpha\geq 0$, we  define  the \emph{n-th $\alpha$-metric} and the \emph{n-th $\alpha$-estimation (Bowen) ball} as
\begin{equation}\label{equ 5.20}
\begin{split}
	\tilde{d}_{n}^{\alpha}(x,y)=&\max_{0\leq i\leq n} e^{i \alpha}\tilde{d}(f^ix,f^iy),\\
	B_{\tilde{d}}(x,n,\alpha,r)=&\{y\in \mathcal X: \tilde{d}(f^ix,f^iy)<e^{-i\alpha}r, 0\leq i\leq n\}
\end{split}
\end{equation}

Unlike classical Bowen balls, the radius of each layer in the $\alpha$-ball exponentially decays with increasing iterations of the system. However, this decay rate is slower than that of a $r$-neutralized ball when $\alpha=r$, which poses a challenge for us to define a two-sided $\alpha$-estimation ball. To give an answer to Question 2, we introduce the following $\alpha$-ball:

\begin{definition}
	For $m,n\in \mathbb N, r>0$, set
	\begin{align*}
		B_{\tilde{d}}(x,-n,m,\alpha,r)=\{y\in \mathcal X: \tilde{d}(f^ix,f^iy)<e^{-|i|\alpha}r, -n\leq i\leq m\}.
	\end{align*}
	We call it as two-sided $\alpha$-estimation ball.
\end{definition}

Denote by
\begin{align*}
	r_{\tilde{d}}^{\alpha}(f,-n,m,r)=\min \left\{\# E\colon \mathcal X=\bigcup_{x\in E} B_{\tilde{d}}(x,n+m,\alpha,r)\right\}.
\end{align*}
Then the upper and lower $\alpha$-estimation topological entropy are defined respectively by
\begin{equation}\label{equ 5.21}
	\begin{split}
		\widetilde{\overline{h}_{est,\tilde{d}}^{\alpha}}(f):=&\lim_{r\to 0}\limsup_{m,n\to\infty}\frac{1}{m+n}\log r_{\tilde{d}}^{\alpha}(f,-n,m,r),\\
		\widetilde{\underline{h}_{est,\tilde{d}}^{\alpha}}(f):=&\lim_{r\to 0}\liminf_{m,n\to\infty}\frac{1}{m+n}\log r_{\tilde{d}}^{\alpha}(f,-n,m,r).
	\end{split}
\end{equation}
For each $\mu\in \mathcal{M}(f)$, similar with the definition in Section \ref{Section 4}, one can  define the upper  $\alpha$-estimation local entropy, Brin-Katok entropy, Katok entropy by using  $\alpha$-estimation ball  and letting $r\to 0$
as $\colon$ $\widetilde{\overline{h}_{\mu,\tilde{d}}^{BK,\alpha}}(f,x)$, $\widetilde{\overline{h}_{\mu,\tilde{d}}^{BK,\alpha}}(f)$,
$\widetilde{\overline{h}_{\mu,\tilde{d}}^{K,\alpha}}(f,\delta)$.
By replacing $\limsup$ with $\liminf$, one can similarly define the lower $\alpha$-estimation entropy.

\begin{remark}
	In contrast to the $r$-neutralized entropy, the $\alpha$-estimation topological entropy is not subject to consideration of its dependence on $r$. However,  it is also dependent on the compatible metric of $\mathcal X$, where
	Kawan \cite[Example 1]{Kaw18} gave a supporting example.
	
\end{remark}

The next Lemma relate the  $\alpha$-ball with 
cylinder set and is a bit different from Lemmas \ref{lem 3.3} and \ref{lem 4.4}.
\begin{lemma}\label{prop 5.3}
	Fix $0\leq \alpha< \min \{\log a, \log b\}$ and $r_3\in (0,1)$.	For any sufficiently small $0<r<r_3$, there exist  positive integers $m_3$ and  $n_3$ satisfying the following statement:
	\begin{itemize}
		\item [\rm (1)] for any $x\in \mathcal X$,
		$$B_{\rho}(x,-(n_3+1),m_3+1,\alpha,r_3)\subseteq B_\rho(x,r)\subseteq B_{\rho}(x,-(n_3-1),m_3-1,\alpha,r_3);$$
		
		\item [\rm (2)]$m_3,n_3\rightarrow\infty$, as $r\rightarrow0$.
		
		\item [\rm (3)] $
		\lim\limits_{r \to 0}\frac{m_3+n_3}{\log \frac{1}{r}}=\frac{1}{\log a+\alpha}+ \frac{1}{\log b+\alpha}.$
	\end{itemize}
\end{lemma}

\begin{proof}
	Following Lemma \ref{lem 3.2}, we easily obtain that
	$$B_{\rho}(x,-n_3,m_3,\alpha,r_3)= C_{-q(e^{-n_3\alpha}r_3)-n_3}^{p(e^{-m_3\alpha}r_3)+m_3}(x).$$
	Similar to the proof of Lemma \ref{lem 4.4}, we have
	$n_3'\in (\frac{\log \frac{1}{r}-\log \frac{1}{r_3}-\log a  }{\log a+\alpha},
	\frac {\log \frac{1}{r}-\log \frac{1}{r_3}+\log a  }{\log a+\alpha})$ and $m_3'\in (\frac{\log \frac{1}{r}-\log \frac{1}{r_3}-\log b  }{\log b+\alpha},
	\frac {\log \frac{1}{r}-\log \frac{1}{r_3}+\log b  }{\log b+\alpha})$ Moreover, $|m_3-m_3'|\leq 1, |n_3-n_3'|\leq 1$. 
	As $0<\alpha<\min\{\log a, \log b\}$, it indicates that 
	for any given $m,n \in \mathbb N$, there exists  an integer $j_1,j_2\in [-1,1]$ such that the  equations
	\begin{align*}
		\left\{\begin{array}{cc}
			m+j_1=p(r)-p(e^{-m\alpha}r_3)\\
			n+j_2=q(r)-q(e^{-n\alpha}r_3)
		\end{array}
		\right.
	\end{align*}
	have positive integer solutions $m_3,n_3$. Similar with Lemma \ref{lem 4.4}, we have
	$$B_{\rho}(x,-(n_3+1),m_3+1,\alpha,r_3)\subseteq B_\rho(x,r)\subseteq B_{\rho}(x,-(n_3-1),m_3-1,\alpha,r_3).$$
	Notice that
	\begin{align*}
		\lim_{m_3 \to \infty}\frac{p(e^{-m_3\alpha}r_3)}{m_3}
		&=\lim_{r \to 0}\frac{p(r)-m_3}{m_3}
		=\frac{\alpha}{\log b},\\
		\lim_{n_3 \to \infty}\frac{q(e^{-n_3\alpha}r_3)}{n_3}
		&=\lim_{r \to 0}\frac{q(r)-n_3}{n_3}
		=\frac{\alpha}{\log a}.
	\end{align*}
	Hence, $$\lim_{r \to 0}\frac{m_3+n_3}{-\log r}=\lim_{r \to 0}\frac{m_3}{p(r)} \lim_{r \to 0}\frac{p(r)}{-\log r}+\lim_{r \to 0}\frac{n_3}{q(r)} \lim_{r \to 0}\frac{q(r)}{-\log r}=\frac{1}{\log a+\alpha}+\frac{1}{\log b+\alpha},$$
	which ends the proof.	
\end{proof}

Then we proceed to prove Theorem \ref{thm 1.5}.

\begin{proof}[Proof of Theorem \ref{thm 1.5}]
	By Lemma \ref{prop 5.3},	for sufficiently small $0<r<r_3<1$ and any $0\leq \alpha< \min \{\log a, \log b\}$, there exist two positive integers $m$ and  $n$ such that
	$$B_{\rho}(x,-(n_3+1),m_3+1,\alpha,r_3)\subseteq B_\rho(x,r)\subseteq B_{\rho}(x,-(n_3-1),m_3-1,\alpha,r_3).$$
	
	Denote $$k=\frac{1}{\log a}+\frac{1}{\log b} \quad \text{and}\quad  k_{\alpha}=\frac{1}{\log a+\alpha}+\frac{1}{\log b+\alpha}.$$	
	Repeat the proof of Theorem	\ref{thm 3.5} and \ref{thm 3.6},
	for any $\mu\in \mathcal M (f)$, we have
	$$\frac{d_{\mu}(x,\tilde{d})}{k_{\alpha}}=\widetilde{\overline{h}_{\mu,\tilde{d}}^{BK,\alpha}}(f,x)=\widetilde{\underline{h}_{\mu,\tilde{d}}^{BK,\alpha}}(f,x), \quad \mu\text{-}a.e.$$
	Hence, if $\mu\in \mathcal E (f)$, it is obviously that for $\mu\text{-}a.e.$ $x\in \mathcal X,$
	\begin{align*}
		&\widetilde{\overline{h}_{\mu,\tilde{d}}^{BK,\alpha}}(f,x)=\widetilde{\underline{h}_{\mu,\tilde{d}}^{BK,\alpha}}(f,x)=\widetilde{\overline{h}_{\mu,\tilde{d}}^{BK,\alpha}}(f)=\widetilde{\underline{h}_{\mu,\tilde{d}}^{BK,\alpha}}(f)\\
		&\widetilde{\overline{h}_{est,\tilde{d}}^{\alpha}}(f)=\widetilde{\underline{h}_{est,\tilde{d}}^{\alpha}}(f)
	\end{align*}
	and their common value satisfy the following equations:
	\begin{align*}
		\dim_{H}(\mu,\tilde{d})&=\widetilde{h_{\mu,\tilde{d}}^{BK,\alpha}}(f)\left(\frac{1}{\log a+\alpha}+\frac{1}{\log b+\alpha}\right),\\
		\dim_{H}(\mathcal X,\tilde{d})&=\widetilde{h_{est,\tilde{d}}^{\alpha}}(f)\left(\frac{1}{\log a+\alpha}+\frac{1}{\log b+\alpha}\right).
	\end{align*}
	Connecting with Theorems \ref{thm 3.5} and \ref{thm 3.6}, we finish the proof.
\end{proof}

Denote by $\widetilde{h_{est,\tilde{d}}^{\alpha}}(f)$ ($\widetilde{h_{\mu,\tilde{d}}^{BK,\alpha}}(f),\widetilde{h_{\mu,\tilde{d}}^{K,\alpha}}(f)$) the common value of $\alpha$-estimation entropy when the upper and lower estimation entropy coincide. Then we obtain the following variational principles.
\begin{corollary}
	Let $\mu\in \mathcal E(f)$ and   $0\leq \alpha< \min\{\log a,\log b\}$. Then for any $\delta \in (0,1)$,  the value of $\alpha$-estimation Katok entropy is independent of the choice of $\delta$. Then
	\begin{align*}
		\widetilde{h_{est,\tilde{d}}^{\alpha}}(f)=&\sup\{\widetilde{h_{\mu,\tilde{d}}^{BK,\alpha}}(f)\colon\mu\in \mathcal E(f)\}
		=\sup\{\widetilde{h_{\mu,\tilde{d}}^{K,\alpha}}(f)\colon\mu\in \mathcal E(f)\}
	\end{align*}
	and
	\begin{align*}
		\widetilde{h_{est,\tilde{d}}^{\alpha}}(f)
		=\sup\{\widetilde{h_{\mu,\tilde{d}}^{BK,\alpha}}(f)\colon\mu\in \mathcal M(f)\}.
	\end{align*}
\end{corollary}

\subsection{Variational principles  for expansive maps}
We are interested in the relationship of $r$-neutralized entropy and $\alpha$-estimation entropy when $\alpha=r$, while Theorems \ref{thm 1.3} and \ref{thm 1.5} give a negative answer. However, it may be posssible to explore the relationship for expansive maps.

If we replace $B_{\tilde{d}}(x,-n,m,\alpha,r)$ by  $B_{\tilde{d}}(x,n,\alpha,r)$  in (\ref{equ 5.21})
and define
\begin{equation}\label{equ 5.22}
	\begin{split}
		\overline{H}_{est,\tilde{d}}^{\alpha}(f):=&\lim_{r\to 0}\limsup_{n\to\infty}\frac{1}{n}\log r_{\tilde{d}}^{\alpha}(f,n,r),\\
		\underline{H}_{est,\tilde{d}}^{\alpha}(f):=&\lim_{r\to 0}\liminf_{n\to\infty}\frac{1}{n}\log r_{\tilde{d}}^{\alpha}(f,n,r),
	\end{split}
\end{equation}
it is uncertain whether (\ref{equ 5.21}) and (\ref{equ 5.22}) are equal since the relationship in Proposition \ref{prop 3.4} is not hold.

Nevertheless, we can obtain some results for a positively expansive map. In this condition, we have
$\rho(x,y)=b^{-n^+(x,y)}$.

\begin{proposition}\label{prop 5.5}
	Let $f$ be a positively expansive map. Given $0\leq \alpha< \log b$. Then for any $\mu \in \mathcal M(f)$ and  $\mu$-a.e. $x\in \mathcal X,$
	\begin{align*}  d_{\mu}(x,\tilde{d})=\overline{H}_{\mu,\tilde{d}}^{BK,\alpha}(f,x)\left(\frac{1}{\alpha+\log b}\right)=\underline{H}_{\mu,\tilde{d}}^{BK,\alpha}(f,x)\left(\frac{1}{\alpha+\log b}\right).
	\end{align*}
	Moreover, if $\mu \in \mathcal E(f)$, then for $\mu$-a.e. $x\in \mathcal X,$
	\begin{align*}
		&\overline{H}_{\mu,\tilde{d}}^{BK,\alpha}(f)=\underline{H}_{\mu,\tilde{d}}^{BK,\alpha}(f)=\overline{H}_{\mu,\tilde{d}}^{BK,\alpha}(f,x)=\underline{H}_{\mu,\tilde{d}}^{BK,\alpha}(f,x),\\
		&\overline{H}_{est,\tilde{d}}^{\alpha}(f)=\underline{H}_{est,\tilde{d}}^{\alpha}(f),
	\end{align*}
	and their common value satisfy the following relations:
	\begin{align*}
		&\dim_{H}(\mu,\tilde{d})=H_{\mu,\tilde{d}}^{BK,\alpha}(f)\left(\frac{1}{\alpha+\log b}\right),\\
		&\dim_{H}(X,\tilde{d})=H_{est,\tilde{d}}^{\alpha}(f)\left(\frac{1}{\alpha+\log b}\right),
	\end{align*}
	where $\overline{H}_{\mu,\tilde{d}}^{BK,\alpha}(f,x),\underline{H}_{\mu,\tilde{d}}^{BK,\alpha}(f,x),\overline{H}_{\mu,\tilde{d}}^{BK,\alpha}(f),\underline{H}_{\mu,\tilde{d}}^{BK,\alpha}(f)$ denote the corresponding local entropy and Brin-Katok entropy with  $B_{\tilde{d}}(x,n,\alpha,r)$.
\end{proposition}

\begin{remark}
	\begin{enumerate}
		\item For a positively expansive map, the $r$-neutralized entropy   is equal to the $\alpha$-estimation entropy if and only if $0<\alpha=r< \frac{1}{\log b}$.
		
		\item Consider an expansive homeomorphism. The $r$-neutralized entropy   is equal to the $\alpha$-estimation entropy with respect to the metric $\tilde{d}$ if and only if
		the following equation has a solution:
		\begin{align}\label{equ 5.23}
			r+\frac{1}{\frac{1}{\log a}+\frac{1}{\log b}}=\frac{1}{\frac{1}{\log a+\alpha}+\frac{1}{\log b+\alpha}},
		\end{align}
		where $0<r<\frac{3}{\frac{1}{\log a}+\frac{1}{\log b}}$ and $0< \alpha<\min\{\log a,\log b\}$.	
		\begin{enumerate}
			\item If $a=b$, then \eqref{equ 5.23} has a unique solution $\Leftrightarrow \alpha=2r<\min\{\log a,\log b\}$.
			
			\item For fixed $\alpha$, the equation has a unique solution $r(\alpha)$.
			
			\item For fixed $r$, the solution exists only and if only
			\begin{align*}
				\alpha=\frac{-t(\log a+\log b)+2+\sqrt{t^2(\log a-\log b)^2+4}}{2t}<\min\{\log a,\log b\},
			\end{align*}
			where $\frac{1}{t}=r+\frac{1}{\frac{1}{\log a}+\frac{1}{\log b}}>0$.
			
		\end{enumerate}
	\end{enumerate}
\end{remark}

\section{Expansive aperiodic zero-dimensional homeomorphism}\label{Section 6}
In this section, we prove that the condition of  $a,b<2$ can be removed  for any zero-dimensional expansive homeomorphisms.

We first propose an example of  symbolic systems with finite symbols and show that  the condition of $a,b<2$ can be dropped under such systems.

Let $E=\{0,1,\cdots,M-1\}$ ($M \geq 2$) be $M$ symbols equipped with discrete
topology. Consider the left shift $\sigma$ on
$\Sigma=\{0,1\cdots,M-1\}^{\mathbb Z} $ with product topology.
Define
\begin{align*}
	n^+(w,\tau)&=\min\{ m\in \mathbb N\colon   w_m\neq\tau_m\},\\
	n^-(w,\tau)&=\min\{ n\in \mathbb N\colon   w_{-n}\neq\tau_{-n}\}.
\end{align*}
Take $a,b>1$  and define $\rho(w,\tau)=\max\{a^{-n^-(w,\tau)},b^{-n^+(w,\tau)}\}.$ Then $\rho$ is a metric compatible with the topology on $\Sigma$.
\begin{proposition}\label{prop 6.1}
	Let $r>0,\alpha\geq 0.$	 Then we have$\colon$
	\begin{enumerate}
		\item \emph{Local entropy and pointwise dimension}: if $\nu\in \mathcal M(\sigma)$, then for $\nu$-a.e. $w\in \Sigma,$
		\begin{align*}
			&d_\nu(w,\rho)=h_{\nu}(\sigma,w)\left(\frac{1}{\log a}+\frac{1}{\log b}\right);\\
			&d_\nu(w,\rho)=\widetilde{h_{\nu,\rho}^{BK,\alpha}}(\sigma,w)\left(\frac{1}{\log a+\alpha}+\frac{1}{\log b+\alpha}\right);\\
			&d_\nu(w,\rho)=h_{\nu,\rho}^{BK,r}(\sigma,w)\left(\frac{1}{r+\frac{1}{\frac{1}{\log a}+\frac{1}{\log b}}}\right).
		\end{align*}

		\item \emph{Brin-Katok entropy and Hausdorff dimension of measures}: if $\nu\in \mathcal E(\sigma)$, then for $\nu$-a.e. $w\in \Sigma,$
		\begin{align*}
			&\dim_H(\nu,\rho)=\dim_B(\nu,\rho)=h_{\nu}(\sigma)\left(\frac{1}{\log a}+\frac{1}{\log b}\right);\\
			&\dim_H(\nu,\rho)=\widetilde{h_{\nu,\rho}^{BK,\alpha}}(\sigma)\left(\frac{1}{\log a+\alpha}+\frac{1}{\log b+\alpha}\right);\\
			&\dim_H(\nu,\rho)=h_{\nu,\rho}^{BK,r}(\sigma)\left(\frac{1}{r+\frac{1}{\frac{1}{\log a}+\frac{1}{\log b}}}\right).
		\end{align*}
		
		\item \emph{Topological entropy and Hausdorff dimension of $\Sigma$:}	
		\begin{align*}
			&\dim_H(\Sigma,\rho)=\dim_B(\Sigma,\rho)=h_{top}(\sigma)\left(\frac{1}{\log a}+\frac{1}{\log b}\right);\\
			&\dim_H(\Sigma,\rho)=\widetilde{h_{top,\rho}^{\alpha}}(\sigma)\left(\frac{1}{\log a+\alpha}+\frac{1}{\log b+\alpha}\right);\\
			&\dim_H(\Sigma,\rho)=h_{top,\rho}^{r}(\sigma)\left(\frac{1}{r+\frac{1}{\frac{1}{\log a}+\frac{1}{\log b}}}\right).
		\end{align*}
		
	\end{enumerate}
\end{proposition}

\begin{proof}
	We only provide the proofs for the first three equations. The remaining equalities can be proved by using a similar approach.
	Let $\alpha\geq 0$  and $r>0. $
	For any $r_0,r_1,r_2\in (0,1)$,  we notice that
	\begin{align*}
		B_{\rho}(w,r_0)=&\left\{\tau\in \Sigma: n^+(w,\tau)>\frac{\log \frac{1}{r_0}}{\log b}, n^-(w,\tau)>\frac{\log \frac{1}{r_0}}{\log a}
		\right\}
		\\=&\left\{\tau\in \Sigma: \tau_i=w_i, -[\frac{\log \frac{1}{r_0}}{\log a}]\leq i\leq [\frac{\log \frac{1}{r_0}}{\log b}]\right\}
	\end{align*}
	Then the three types of balls are expressed as:	
	\begin{align*}	
		&B_\rho(w,-n,m,e^{-(n+m)r})\\
		=&\left\{\tau\in \Sigma: \tau_i=w_i, -[\frac{(n+m)r}{\log a}]-n\leq i\leq [\frac{(n+m)r}{\log b}]+m\right\};\\
		&B_\rho(w,-n_1,m_1,r_1)\\
		=&\left\{\tau\in \Sigma: \tau_i=w_i, -[\frac{\log \frac{1}{r_1}}{\log a}]-n_1\leq i\leq [\frac{\log \frac{1}{r_1}}{\log b}]+m_1\right\}	;\\
		&B_\rho(w,-n_2,m_2,\alpha,r_2
		)\\
		=&\left\{\tau\in \Sigma: \tau_i=w_i, -[\frac{n_2\alpha+\log\frac{1}{r_2}}{\log a}]-n_2\leq i\leq [\frac{m_2\alpha+\log\frac{1}{r_2}}{\log b}]+m_2\right\}.
	\end{align*}
	
	Let $\nu\in \mathcal M(\sigma)$. Define $$k=\frac{1}{\log a}+\frac{1}{\log b} \quad \text{and}\quad  k_{\alpha}=\frac{1}{\log a+\alpha}+\frac{1}{\log b+\alpha}.$$
	\begin{enumerate}		
		\item Linking open Balls with Bowen balls by taking
		$$m_1=[\frac{-\log r_0}{\log b}]-[\frac{-\log r_1}{\log b}],\quad n_1=[\frac{-\log r_0}{\log a}]-[\frac{-\log r_1}{\log a}].$$
		Then $B_{\rho}(w,r_0)=B_\rho(w,-n_1,m_1,r_1)$.
		Similar with Theorem \ref{thm 3.5}, for $\nu$-a.e. $w\in \Sigma$, we have
		\begin{align*}
			d_\nu(w,\rho)=\lim_{r_1 \to 0}\frac{\log \mu_{\rho}(w,r)}{\log r_0}
			=&\lim_{m_1,n_1 \to \infty}\frac{\log \nu(B_\rho(w,-n_1,m_1,r_1))}{m_1+n_1}\cdot \lim_{r_1 \to 0}\frac{m_1+n_1}{\log r_0}\\
			=&kh_{\nu}(\sigma,w).
		\end{align*}
		
		\item Linking Bowen balls with $r$-neutralized balls by taking
		$$m_1=[\frac{(n+m)r}{\log b}]+m-[\frac{-\log r_1}{\log b}],\quad n_1=[\frac{(n+m)r_1}{\log a}]+n-[\frac{-\log r_1}{\log a}].$$
		Then $B_\rho(w,-n_1,m_1,r_1)=B_\rho(w,-n,m,e^{-(n+m)r})$. For $\nu$-a.e. $w\in \Sigma$,
		\begin{align*}
			\overline{h}_{\nu,\rho}^{BK,r}(\sigma,w)&=\underline{h}_{\nu,\rho}^{BK,r}(\sigma,w)=h_{\nu}(\sigma,w)\lim_{m,n\to \infty}\frac{m_1+n_1}{m+n}= (1+rk)h_{\nu}(\sigma,w),
		\end{align*}
		
		\item Linking open balls with $\alpha$-estimation balls by taking
		\begin{align*}
			[\frac{\log \frac{1}{r_0}}{\log b}]=[\frac{m_2\alpha-\log r_2}{\log b}]+m_2,\quad [\frac{\log \frac{1}{r_0}}{\log a}]=[\frac{n_2\alpha-\log r_2}{\log a}]+n_2,
		\end{align*}
		respectively.
		Then there exists $m_2,n_2 \in \mathbb N$ such that  $$B_\rho(w,-(n_2+1),m_2+1,\alpha,r_2
		)\subseteq B_{\rho}(w,r)\subseteq B_\rho(w,-(n_2-1),m_2-1,\alpha,r_2
		).$$ 	
		Hence, for $\nu$-a.e. $w\in \Sigma$,	\begin{align*}
			\widetilde{\overline{h}_{\nu,\rho}^{BK,\alpha}}(\sigma,w)=\widetilde{\underline{h}_{\nu,\rho}^{BK,\alpha}}(\sigma,w)=d_\nu(w,\rho)\lim_{r_0}\frac{-\log r_0}{n_2+m_2} =\frac{d_\nu(w,\rho)}{k_\alpha} ,
		\end{align*}
	\end{enumerate}
	which finishes the proof.
\end{proof}

Actually, for a general expansive homeomorphism, if  following the approach of the construction of  hyperbolic metrics in Section \ref{Section 2}, we are failed to drop the  condition of $a,b<2$. (See Remark \ref{rem 2.9}.)
However, for some certain expansive systems, for instance called as zero-dimensional aperiodic expansive homeomorphism, we can remove this condition.

For example, any minimal subshift of finite symbols is a zero-dimensional aperiodic expansive homeomorphism.
On the study of the zero-dimensional aperiodic expansive homeomorphism,  two important tools are symbolic systems and  topological universality.
\begin{definition}
	Let $\mathcal C$ be a collection of expansive
	aperiodic zero-dimensional systems. A topological system $(\mathcal Y,S)$ is said to be topologically $\mathcal C$-universal if for any $(\mathcal X, f)\in \mathcal C$, there exists a subsystem $(\mathcal Z,R)$ of $(\mathcal Y,S)$ topologically conjugated to $(\mathcal X, f)$.
\end{definition}

Regarding topological universality, Krieger achieved an important result in \cite{Kri82}, which established a connection between the full shift and topological universality. Denote by
$$\{h_{top} < \log M\}=\{(\mathcal X,f)\in \mathcal C\colon h_{top}(f) < \log M
\}.$$
\begin{lemma}[\cite{Kri82}]\label{lem 6.2}
	Let $(\Sigma,\sigma)$ be a full shift of $M$-symbols. Then $(\Sigma,\sigma)$ is topologically
	$\{h_{top} < \log M\}$-universal.
\end{lemma}

Then we can obtain a stronger result for 
the expansive
aperiodic zero-dimensional homeomorphism.
\begin{theorem}
	Let $(\mathcal X, f)$ be an expansive
	aperiodic zero-dimensional homeomorphism. Then for any $a,b>1$, there exists a metric $\tilde{d}$ defining its topology such that
	\begin{enumerate}
		\item if
		$\mu\in \mathcal M(f)$, then for $\mu$-a.e. $x\in \mathcal X,$
		\begin{equation}\label{equ 6.1}
			\begin{split}
				d_{\mu}(x,\tilde{d})=&h_{\mu}(f,x)\left(\frac{1}{\log a}+\frac{1}{\log b}\right);\\	
				d_{\mu}(x,\tilde{d})=&\widetilde{h_{\mu,\tilde{d}}^{BK,\alpha}}(f,x)\left(\frac{1}{\log a+\alpha}+\frac{1}{\log b+\alpha}\right)\\
				d_{\mu}(x,\tilde{d})=&h_{\mu,\tilde{d}}^{BK,r}(f,x)\left(\frac{1}{r+\frac{1}{\frac{1}{\log a}+\frac{1}{\log b}}}\right).
			\end{split}
		\end{equation}
		
		\item If $\mu\in \mathcal E(f)$, then for $\mu$-a.e. $x\in \mathcal X,$
		\begin{equation}\label{equ 6.2}
			\begin{split}	
				\dim_{H}(\mu,\tilde{d})=&h_{\mu}(f)\left(\frac{1}{\log a}+\frac{1}{\log b}\right);\\	
				\dim_{H}(\mu,\tilde{d})=&\widetilde{h_{\mu,\tilde{d}}^{BK,\alpha}}(f)\left(\frac{1}{\log a+\alpha}+\frac{1}{\log b+\alpha}\right)\\
				\dim_{H}(\mu,\tilde{d})=&h_{\mu,\tilde{d}}^{BK,r}(f)\left(\frac{1}{r+\frac{1}{\frac{1}{\log a}+\frac{1}{\log b}}}\right).
			\end{split}
		\end{equation}
		\item
		\begin{equation}\label{equ 6.3}
			\begin{split}
				\dim_{H}(X,\tilde{d})=&h_{top}(f)\left(\frac{1}{\log a}+\frac{1}{\log b}\right);\\	
				\dim_{H}(X,\tilde{d})=&\widetilde{h_{top,\tilde{d}}^{\alpha}}(f)\left(\frac{1}{\log a+\alpha}+\frac{1}{\log b+\alpha}\right)\\
				\dim_{H}(X,\tilde{d})=&h_{top,\tilde{d}}^{r}(f)\left(\frac{1}{r+\frac{1}{\frac{1}{\log a}+\frac{1}{\log b}}}\right).
			\end{split}
		\end{equation}
	\end{enumerate}
\end{theorem}

\begin{proof}
	Since $f$ is an expansive homeomorphism, then there exists $M\in \mathbb N$ such that $h_{top}(f)<\log M$.
	Let $(\Sigma,\sigma)$ be the full shift of $M$ symbols given in Proposition \ref{prop 6.1}.
	By Lemma \ref{lem 6.2}, there exists a subsystem $(\Sigma',\sigma)$ of $(\Sigma,\sigma)$ topologically conjugated to $(\mathcal X, f)$ by the homeomorphism $\pi\colon \mathcal  X\rightarrow \Sigma'$. Let $\rho$ be the metric defined in Proposition \ref{prop 6.1}. Then the equalities (\ref{equ 6.1}),(\ref{equ 6.2}),(\ref{equ 6.3}) also hold for the subsystem
	$(\Sigma',\sigma)$ and any invariant measures $\nu\in \mathcal M(\sigma)$.
	
	For any $x,y\in \mathcal X$, we define the metric on $\mathcal X$ by
	$$\tilde{d}(x,y)=\rho(\pi x,\pi y).$$
	which is compatible with the topology of $\mathcal{X}$. Since for any $\mu \in \mathcal{M}(f)$,  we have $\pi_\ast\mu \in \mathcal{M}(\sigma)$. Then for $\mu$-a.e. $x \in \mathcal{X}$ and $\pi_\ast\mu $-a.e. $w \in \Sigma'$,
	by the relation
	$B_{\tilde{d}}(x, r) = \pi^{-1} B_{\rho}(\pi x, r)$,
	we have
	\begin{align*}
		h_{\mu}(f) &= h_{\mu \circ  \pi}(\sigma, \Sigma'), \quad \dim_{H}(\mu, \tilde{d}) = \dim_{H}(\pi_\ast\mu, \rho), \quad d_{\mu}(x,\tilde{d}) = d_{\pi_\ast\mu}(w,\rho); \\
		h_{\text{top}}(f) &= h_{\text{top}}(\sigma, \Sigma'), \quad \dim_{H}(f, \tilde{d}) = \dim_{H}(\Sigma', \rho).
	\end{align*}
	Furthermore,  $(\mathcal{X}, f)$ and $(\Sigma', \sigma)$ share the same $r$-neutralized entropy and $\alpha$-estimation entropy, which finishes the proof.
\end{proof}

\begin{open question}
	For a general expansive homeomorphism, whether 
	there exists another method to construct compatible metrics such that 
	it is possible to drop the condition of $a,b<2$.
\end{open question}

\section*{Acknowledgments}
The first author was supported by the National Natural Science Foundation of China (Nos.12471184 and 12071222). The work was also
funded by the Priority Academic Program Development of Jiangsu
Higher Education Institutions. We would like to
express our gratitude to Tianyuan Mathematical Center in Southwest China (No.11826102), Sichuan University and Southwest Jiaotong University for their support and hospitality.


\begin{thebibliography}{MM}

\bibitem{BPS99} L. Barreira, Y. Pesin  and  J. Schmeling, \textit{Dimension and product structure of hyperbolic measures}, Ann. of Math.(2)  {\bf 149}(1999),  no.3, 755-783.

\bibitem{Bow72} R. Bowen and P. Walters, \textit{Expansive one-parameter flows}, J. Differential Equations {\bf 12}(1972),180-193.

\bibitem{Bow75} R. Bowen, \textit{Equilibrium States and the Ergodic Theory of Anosov Diffeomorphisms}, Lecture Notes in Math., Vol. 470, Springer, Berlin, 1975.

\bibitem{BK83} M. Brin and A. Katok, \textit{On local entropy}, Geometric Dynamics (Rio de Janeiro, 1981). Lecture Notes in Math., Vol. 1007, Springer, Berlin, 1983, 30-38.


\bibitem{CX97} E. Chen and J. Xiong, \textit{Dimension and measure theoretic entropy of a subshift in symbolic space}, Chinese Sci. Bull.  {\bf 42}(1997), no. 14, 1193-1196.

\bibitem{CF98} H. Crauel and F. Flandoli,
\textit{Hausdorff dimension of invariant sets for random dynamical systems}, J. Dynam. Differential Equations {\bf 10}(1998), no. 3, 449-474.


\bibitem{DZG98} X. Dai, Z. Zhou and X. Geng, \textit{Some relations between Hausdorff-dimensions and entropies}, Sci. China Ser. A {\bf 41}(1998), no. 10, 1068-1075.



\bibitem{DQ25} C. Dong and Q. Qiao, \textit{On $r$-neutralized entropy: entropy formula and existence of measures attaining the supremum}, 
Com. Math. Phys. {\bf 406}(2025), no. 4, Paper No. 74.

\bibitem{ER85} J. -P. Eckmann and D. Ruelle, 
\textit{Ergodic theory of chaos and strange attractors},
Rev. Modern Phys. {\bf 57}(1985), no. 3, 617-656.


\bibitem{Fat89} A. Fathi, \textit{Expansiveness, hyperbolicity and Hausdorff dimension}, Com. Math. Phys. {\bf 126}(1989), no. 2, 249-262.

\bibitem{FH12} D. Feng and W. Huang, \textit{Variational principles for topological entropies of subsets}, J. Funct. Anal.   {\bf 263}(2012), no. 8, 2228-2254.

\bibitem{Fried 87} D. Fried, \textit{Finitely presented dynamical systems}, Ergodic Theory Dynam. Systems {\bf 7}(1987), 489-507.

\bibitem{Fri37} A. Frink, \textit{Distance functions and the metrization problem}, Bull. Amer. Math. Soc. {\bf 43}(1937), no. 2, 133-142.

\bibitem{GK20} V. Gundlach and Y. Kifer,
\textit{Expansiveness, specification, and equilibrium states for random bundle transformations}, Discrete Contin. Dynam. Systems {\bf 6}(2000), no. 1, 89-120.

\bibitem{GWZ23} Y. Guo, X. Wang and Y. Zhu,
\textit{Forward expansiveness and entropies for subsystems of $\mathbb Z_+^k$-actions}, Acta Math. Sin. (Engl. Ser.) {\bf 39}(2023), no. 4,  633-662.

\bibitem{Kat80} A. Katok, \textit{Lyapunov exponents, entropy and periodic orbits for diffeomorphisms}, Publ. Math. Inst. Hautes \'Etudes Sci. \textbf{51}(1980), 137-173.

\bibitem{Kaw18} C. Kawan, \textit{Exponential state estimation, entropy and Lyapunov exponents}, Systems Control Lett. {\bf 113}(2018), 78-85.

\bibitem{Kri82} W. Krieger,
\textit{On the subsystems of topological Markov chains}, Ergodic Theory Dynam. Systems  {\bf 2}(1982), no. 2, 195-202.



\bibitem{LY85a} F. Ledrappier and L.-S. Young, \textit{The metric entropy of diffeomorphisms. I. Characterization of measures satisfying Pesin’s entropy formula}, Ann. of Math.(2) {\bf 122}(1985), no. 3, 509-539.

\bibitem{LY85b} F. Ledrappier and L.-S. Young, \textit{The metric entropy of diffeomorphisms. II. Relations between
	entropy, exponents and dimension}, Ann. of Math.(2)  {\bf 122}(1985), no. 3, 540-574.

\bibitem{LT18} E. Lindenstrauss  and M. Tsukamoto, \textit{From rate distortion theory to metric mean dimension: variational principle}, IEEE Trans. Inform. Theory \textbf{64}(2018), 3590-3609.



\bibitem{Man79} R. Ma\~{n}\'{e},  \textit{Expansive homeomorphisms and topological dimension}, Trans. Amer. Math. Soc. {\bf 252}(1979), 313-319.

\bibitem{Mat82} J. Mather,  \textit{Existence of quasiperiodic orbits for twist homeomorphisms of the annulus}, Topology {\bf 21}(1982), no.4, 457-467.

\bibitem{MT19} T. Meyerovitch and M. Tsukamoto,
\textit{Expansive multiparameter actions and mean dimension}, Trans. Amer. Math. Soc. {\bf 371}(2019), no.10, 7275-7299.

\bibitem{Ova24}  B. Ovadia, \textit{Tubular dimension: leaf-wise asymptotic local product structure, and entropy and volume growth}, arXiv:2402.02496v2, 2024.

\bibitem{OR24} B. Ovadia and F. Rodriguez-Hertz,
\textit{Neutralized local entropy and dimension bounds for invariant measures}, Int. Math. Res. Not. IMRN (2024), no. 11, 9469-9481.

\bibitem{PV20} M. Pacifico and J. Vieitez,
\textit{Lyapunov exponents for expansive homeomorphisms}, Proc. Edinb. Math. Soc.(2) {\bf 63}(2020), no.2, 413-425.

\bibitem{Pes77}Y. Pesin, \textit{Characteristic Lyapunov exponents and smooth ergodic theory}, Russian Math. Surveys {\bf 32}(1977), no.4, 55-112, 287.

\bibitem{Pes97} Y. Pesin, \textit{Dimension Theory in Dynamical Systems}, University of Chicago Press, Chicago, IL, 1997.

\bibitem{RR20} S. Roth and Z. Roth, \textit{Inequalities for entropy, Hausdorff dimension, and Lipschitz constants}, Studia Math. {\bf 250}(2020), no. 3, 253-264.

\bibitem{Sch06} V. Schroeder,
\textit{Quasi-metric and metric spaces}, Conform. Geom. Dyn. {\bf 10}(2006), 355-360.


\bibitem{Thi91} P. Thieullen,
\textit{G\'en\'eralisation du th\'eor\`eme de {P}esin pour
	l'{$\alpha$}-entropie}, Lyapunov exponents (Oberwolfach, 1990). Lecture Notes in Math., Vol. 1486, Springer, Berlin, 1991, 232-242.

\bibitem{URM22}M. Urba\'nski, M. Roy and S. Munday,  \textit{Non-invertible dynamical systems. Vol. 1. Ergodic theory—finite and infinite, thermodynamic formalism, symbolic dynamics and distance expanding maps}, De Gruyter, Berlin,  2022.


\bibitem{Wal82} P. Walters,  \textit{An Introduction to Ergodic Theory}, Springer, Berlin, 1982.



\bibitem{YCZ24} R. Yang, E. Chen and X. Zhou, \textit{Variational principle for neutralized Bowen topological entropy}, Qual. Theory Dyn. Syst. {\bf 23}(2024), Paper No. 162, 15 pp.

\bibitem{You82} L.-S. Young,  \textit{Dimension, entropy and Lyapunov exponents}, Ergodic Theory Dynam. Systems {\bf 2}(1982), no. 1, 109-124. 		

\bibitem{ZC21} X. Zhong and Z. Chen,  \textit{A variational principle for Bowen estimation entropy}, Sci. Sin. Math.(in Chinese)	 {\bf 51}(2021)
1213-1224.


\end{thebibliography}
\end{document}